\documentclass{amsart}

\usepackage{latexsym}
\usepackage{amsfonts}
\usepackage{amssymb}
\usepackage{rotating}
\usepackage{enumerate}
\usepackage{amsmath}
\usepackage{amscd}
\usepackage{amsthm}
\usepackage{xypic}

\def\ints{{\mathbb Z}}
\def\rats{{\mathbb Q}}
\def\nats{{\mathbb N}}
\def\reals{{\mathbb R}}
\def\complex{{\mathbb C}}

\def\proj{{\mathbb P}}
\def\aff{{\mathbb A}}
\def\FF{{\mathbb F}}
\def\Gp{{\mathbb G}}

\def\qed{\hfill $\Box$}

\def\Hom{{\text{Hom}}}

\def\cf{{\text{char}}}
\def\Frac{{\text{Frac}}}
\def\Sets{{\text{Sets}}}

\def\top{{\text{top}}}

\def\an{{\text{an}}}

\DeclareMathOperator{\Aut}{Aut}

\def\Ind{{\text{Ind}}}

\DeclareMathOperator{\Spec}{Spec }
\def\Spf{{\mbox{Spf }}}

\def\mc#1{\mathcal{#1}}
\def\mf#1{\mathfrak{#1}}
\def\ol#1{\overline{#1}}

\newtheorem{theorem}{Theorem}[section]
\newtheorem{prop}[theorem]{Proposition}
\newtheorem{lemma}[theorem]{Lemma}
\newtheorem{conjecture}[theorem]{Conjecture}
\newtheorem{corollary}[theorem]{Corollary}
\newtheorem{predefinition}[theorem]{Definition}
\newenvironment{definition}{\begin{predefinition}\rm}{\end{predefinition}}
\newtheorem{preremark}[theorem]{Remark}
\newenvironment{remark}{\begin{preremark}\rm}{\end{preremark}}
\newtheorem{preconstruction}[theorem]{Construction}

\newtheorem{preframework}[theorem]{Framework}

\newtheorem{prenotation}[theorem]{Notation}

\newtheorem{preexample}[theorem]{Example}
\newenvironment{example}{\begin{preexample}\rm}{\end{preexample}}
\newtheorem{preclaim}[theorem]{Claim}

\newtheorem{prequestion}[theorem]{Question}
\newenvironment{question}{\begin{prequestion}\rm}{\end{prequestion}}

\numberwithin{equation}{section}
\begin{document}

\title[Lifting of curves with automorphisms]{Lifting of curves with automorphisms} 

\author{Andrew Obus}
\address{University of Virginia, 141 Cabell Drive, Charlottesville, VA
22904}
\email{andrewobus@gmail.com}
\thanks{The author was supported by NSF Grants DMS-1265290 and DMS-1602054}
{\let\thefootnote\relax\footnotetext{\emph{2010 Mathematics Subject
    Classification: Primary 14H37, 12F10; Secondary 11G20, 12F15, 13B05, 13K05, 14G22, 14H30.}}}
\date{\today}

\keywords{branched cover, lifting, Galois group, Oort conjecture, wild
ramification, curves with automorphisms}

\begin{abstract}
The lifting problem for curves with automorphisms asks whether we can lift a smooth projective characteristic $p$
curve with a group $G$ of automorphisms to characteristic zero.  This was solved by
Grothendieck when $G$ acts with prime-to-$p$ stabilizers, and there has
been much progress over the last few decades in the wild case.  We
survey the techniques and obstructions for this lifting problem,
aiming at a reader whose background is limited to scheme
theory at the level of Hartshorne's book.  Throughout, we include
numerous examples and clarifying remarks.  We also provide a list of
open questions.
\end{abstract}

\maketitle

\section{Introduction}\label{Sintro}
A fundamental success of algebraic geometry its ability to reason
about algebraically defined objects in characteristic $p$ using geometric
intuition gleaned from our experience in the real world.  For instance,
not only can we define concepts such as \emph{smoothness} and
\emph{tangent spaces} in characteristic $p$, but such concepts also
turn out to reflect our characteristic zero intuition surprisingly
well.  

Of course, characteristic $p$ geometry does not behave \emph{exactly}
like geometry in characteristic zero!   The differences are too numerous to count,
but let us quickly mention one example, which will be motivating for
the discussion to come.  Consider the affine line
$X = \aff^1_{\complex}$.  There are certainly no nontrivial finite
\'{e}tale covers\footnote{For our purposes, a \emph{branched cover} $f: Y \to X$ is a
  finite, flat, generically \'{e}tale morphism of geometrically connected,
  normal schemes.  If $f$ is unramified, we say it is an
  \emph{\'{e}tale} cover.}
$Y \to X$, because any such cover would give a topological cover
$Y(\complex) \to X(\complex)$ with the complex topology (e.g.,
\cite[Proposition 4.5.6]{Sz:gg}), contradicting
the fact that $X(\complex) \cong \complex$ is simply
connected. However, in characteristic $p$ we have the following example.

\begin{example}\label{Ebasicwild}
Let $k$ be a field of characteristic $p$ (say, algebraically closed, to preserve the analogy
as well as possible).  The map from the zero-locus $Y$ of $y^p -
y = x$ in $\aff^2_k$ to $\aff^1_k$ given by projecting to the
$x$-coordinate is a finite \'{e}tale cover.  Indeed, the group
$\ints/p$ acts freely on $Y$ by sending $(x, y)$ to $(x, y+1)$, and
the projection is nothing but the quotient map.
\end{example}

In fact, there are a wealth of finite \'{e}tale covers of $\aff^1_k$.  The famous
\emph{Abhyankar's conjecture} (now a result of Raynaud (\cite{Ra:ab})
and later generalized by Harbater (\cite{Ha:ac})) shows that for every
$G$ generated by its $p$-Sylow subgroups (in particular, for every
simple group $G$ with order divisible by $p$), there exists a finite
\'{e}tale $G$-Galois cover $Y \to \aff^1_k$.  

The lifting problem and local lifting problem that are the subject of
this paper are inspired by understanding the relation between branched
covers in characteristic $p$ and in characteristic zero.  

\subsection{Riemann's existence theorem}\label{SRiemann}

Our starting point is the following theorem (see {\cite[Theorem 2.1.1]{Ha:pg}).

\begin{theorem}{\rm \textbf{(Riemann's existence theorem)}}\label{TRiemann}
Let $X$ be a smooth, projective algebraic curve over $\complex$, and let $f^{\top}:
Y^{\top} \to X(\complex)$ be a finite topological 
branched cover\footnote{Recall that a \emph{topological branched
    cover} is a continuous map that is a topological cover away from a nowhere
  dense set of points.} with finitely many branch points.
Then there exists a branched cover $f: Y \to X$, unique up to isomorphism, such that the
corresponding topological cover $f^{\an}: Y(\complex) \to X(\complex)$ (in the
complex topology) is isomorphic to $f^{\top}$ (that is, there is
a homeomorphism $i_Y: Y(\complex) \to Y^{\top}$ such that $f^{\top} \circ i_Y = f^{\an}$).
\end{theorem}

Since topological covers correspond to quotients of fundamental
groups, and fundamental groups of Riemann surfaces are well understood, 
this leads to combinatorial/group-theoretic parameterizations of
branched covers of algebraic curves over $\complex$.  
For instance, degree $n$ covers of 
$\proj^1_{\complex}$ branched at $r+1$ points $\{x_0, \ldots, x_r\}$ correspond to index $n$
subgroups of the topological fundamental group of $\proj^1 \backslash
\{x_0, \ldots, x_r\}$, which is the free group $F_r$.  

\begin{example}\label{EBelyi} {\textbf{(Belyi maps)}}
A famous result of Belyi (\cite{Be:ge}) states that every algebraic
curve $X$ defined over $\ol{\rats}$ can be expressed as a branched
cover of $\proj^1_{\ol{\rats}}$, \'{e}tale outside $\{0, 1,
\infty\}$.  Such covers are called \emph{three-point covers}.  By
Riemann's existence theorem, these three-point covers correspond to finite
index subgroups $N$ of $F_2$.  The cover is \emph{Galois} with group
$G$ if and only $N$ is normal in $F_2$ and $F_2/N \cong G$.
\end{example}

\begin{remark}\label{Rcharpbelyi}
The situation in characteristic $p$ is quite different. In fact,
\emph{every} smooth geometrically connected curve $X$ defined over a
perfect field $k$ of characteristic $p$ has a map
to $\proj^1_k$ \'{e}tale outside $\{\infty\}$, see
\cite[Lemma 16]{Ka:TL}, or \cite[Remark 4]{Ab:ca} for the case when $k$
is algebraically closed.  In fact, if there exists an exact
differential form on $X$ whose divisor is supported at one $k$-point
and $k$ is algebraically closed, then
one can force the map to have only one ramification point over
$\infty$ (\cite[Theorem 1]{Za:1p}). 
\end{remark}

The earliest results on the lifting problem were motivated by the
search for a characteristic $p$ version of Riemann's existence
theorem.  That is, is there a way to parameterize branched covers of
curves in characteristic $p$ in terms of well-understood group theory?
In some sense, obtaining a full answer is
hopeless.  For instance, even if we restrict ourselves to branched
covers of $\proj^1$ branched only at $\infty$, there are infinitely
many linearly disjoint $\ints/p$-covers (take the smooth projective
completion of the affine cover $V(y^p - y  - x^N) \subseteq \aff^2 \to
\aff^1$ given by projecting to the $x$-axis, as $N$ ranges through
$\nats \backslash p\nats$).  This tells us that the fundamental group
of $\aff^1_k$, for $k$ an algebraically closed field of
characteristic $p$, is not finitely generated (since it has an
infinitely large elementary abelian $p$-quotient).

However, the situation is better when we restrict our attention to
\emph{tame} covers, that is, covers where $p$ does not divide the
ramification indices.  In this situation, Grothendieck showed that
each branched cover in characteristic $p$ is the reduction of a
branched cover in characteristic zero, which is more or less unique.

\begin{theorem}\label{Ttamelifting} {\rm \textbf{(Grothendieck's
      ``tame Riemann existence converse'' theorem in 
characteristic $p$)}} 
Let $R$ be a complete discrete valuation ring with algebraically closed
residue field $k$ of characteristic $p$
and fraction field $K$ with algebraic closure $\ol{K}$.  Let $X_R$ be
a smooth, projective, relative $R$-curve with special fiber $X$ and
generic fiber $X_K$, and let $x_{1, R}, \ldots, x_{n,R}$ be pairwise disjoint
sections of $X_R \to \Spec R$.  Write $x_i$ for the intersection
of $x_{i,R}$ with $X$.

If $f: Y \to X$ is a \emph{tamely} ramified finite cover, \'{e}tale
above $X \backslash \{x_1, \ldots, x_n\}$, then there is a
unique finite flat branched cover $f_R: Y_R \to X_R$, \'{e}tale above $X_R
\backslash \{x_{1, R}, \ldots, x_{n, R}\}$ with the same ramification
indices as that of $f$, such that the special fiber of $f_R$ is $f$.  If
$f$ is $G$-Galois, then so is $f_R$.
\end{theorem}

\begin{proof}
If $f$ is \emph{\'{e}tale}, the theorem follows from Grothendieck's theory of \'{e}tale lifting (\cite[I, Corollaire 8.4]{SGA1}, combined with \cite[III, Prop.\ 7.2]{SGA1}).
In general, we can use Grothendieck's theory of tame lifting
(\cite[XIII, Corollaire 2.12]{SGA1}, or \cite{We:dt} for an
exposition).  See also \cite{Fu:hs} for an alternate proof.  If $f$ is
a Galois cover, the theorem
also follows from the \emph{local-global principle} (Theorem
\ref{Tlocalglobal}), along with Example \ref{Elocaltame}.  
\end{proof}

\begin{remark}
The cover $f_R$ is called a \emph{lift} of the cover $f$ over $R$.
Thus, the tame Riemann existence converse can be stated succinctly as
``tame covers lift over $R$ (uniquely once the branch locus is
fixed).''  See \S\ref{Slifting}.
\end{remark}

\begin{remark}
Assume without loss of generality that $k$ is countable (any morphism
of varieties in characteristic $p$, being given by finitely many equations, can be defined over some finitely
generated extension of $\FF_p$, whose algebraic closure is thus
countable).  If $R$ is a complete discrete valuation ring with residue
field $k$ and fraction field $K$, then it is an easy exercise to show that $|R| =
|K| = |\ol{K}| = |\complex|$, and thus that $\text{char}(R) = 0$ implies $\ol{K} \cong \complex$,
since they are both algebraically closed fields of characteristic zero
with the same cardinality.  

Suppose $\text{char}(R) = 0$ above.  Then Theorem \ref{Ttamelifting}
shows that once the branch locus is fixed, each tame branched cover in
characteristic $p$ gives rise to a branched cover over $R$,
thus a branched cover over $\ol{K} \cong \complex$ via base
change, and thus to a topological branched cover of $X(\complex)$.
This is why we call the theorem a Riemann
existence converse theorem: It shows that any tame algebraic cover in characteristic
$p$ comes from a topological cover over $\complex$.
Example \ref{Ebasicwild} shows that this is not true for wild covers!
\end{remark}

\begin{remark}\label{Rfgmotivation}
Theorem \ref{Ttamelifting} has the following interpretation in terms of fundamental
groups.  Write $U = X \backslash \{x_1, \ldots, x_n\}$ and write $U_{\ol{K}}
= (X_R \backslash \{x_{1,R}, \ldots, x_{n, R} \})
\times_R \ol{K}$.  Let $\pi_1^t(U)$ be the \emph{tame fundamental
group} of $U$ (that is, the automorphism group of the pro-universal tame cover of $X$, \'{e}tale above $U$), 
and let $\pi_1^t(U_{\ol{K}})$ be the \emph{$p$-tame fundamental group} of
  $U_K \times_K \ol{K}$ (same definition, but replace
``tame" with ``ramified of prime-to-$p$ index").  Then there is a
natural surjection $\phi: \pi_1^t(U_{\ol{K}}) \to \pi_1^t(U),$ well
defined up to conjugation. See \cite[Remark 2.4]{Ob:ll} for details on
this surjection.

In particular, when $\text{char}(R) = 0$, this shows that the tame
fundamental group of a characteristic $p$ curve is a \emph{quotient} of the
$p$-tame fundamental group of a characteristic zero curve of the same
genus with the same number of missing points.
\end{remark}

\begin{remark}
The surjection $\phi$ from Remark \ref{Rfgmotivation} descends to an
\emph{isomorphism} $$\phi': \pi_1^{(p')}(U_{\ol{K}}) \to
\pi_1^{(p')}(U),$$ where $\pi_1^{(p')}$ represents the maximal
prime-to-$p$ quotient of $\pi_1$ (\cite[XIII, Corollaire 2.12]{SGA1}).  This could be called the
``prime-to-$p$ Riemann existence theorem'' in characteristic $p$, as
$\phi'$ gives a bijection on finite-index normal subgroups, thus
placing the prime-to-$p$ Galois covers of $X_{\ol{K}}$,
\'{e}tale above $U_{\ol{K}}$, in natural one-to-one correspondence with
the prime-to-$p$ Galois covers of $X$, \'{e}tale above $U$. 
\end{remark}

Of course, we would like to know whether \emph{wild} (i.e., not tame) covers
in characteristic $p$ come from characteristic zero as well.  This
will be our main focus.

\subsection{Lifting of covers of curves}\label{Slifting}
Let $k$ be an algebraically closed field of characteristic $p$.  We formulate the lifting problem precisely:
\begin{question}{\rm \textbf{(The lifting problem for covers of curves)}}\label{Qgloballifting}
Let $G$ be a finite group acting on a smooth, projective curve $Y$
over $k$, so that $f: Y \to X = Y/G$ is a branched $G$-cover of smooth, projective
curves over $k$.  Is there a characteristic zero discrete valuation ring $R$ with residue
field $k$ and a branched 
$G$-cover $f_R: Y_R \to X_R$ of smooth (in particular, flat) relative $R$-curves with special fiber $f$? 
\end{question}
If the answer to the question above is ``yes," then we say that $f$
\emph{lifts to characteristic zero}  (or \emph{lifts over $R$}), and
that $f_R$ is a \emph{lift} of $Y$ (with $G$-action).

In what follows, we will refer to the lifting problem for covers of
curves simply as ``the lifting problem.''
\begin{remark}
It is clearly equivalent to ask whether $Y$ lifts to a curve $Y_R$
with an action of $G$ by $R$-automorphisms whose restriction to the
special fiber is the original $G$-action on $Y$.  This is how the
lifting problem was originally formulated in
\cite{Oo:la}.  One speaks of $(Y, G)$ lifting to characteristic zero.
\end{remark}

\begin{remark}\label{Rstupid}
Of course, one cannot solve the lifting problem by simply taking
equations for $f$, $Y$, $X$, and the $G$-action and lifting the
coefficients to $R$.  There is no reason to expect that you end up with a
$G$-action on $Y_R$ at all in this case.  For instance, suppose $Y =
\proj^1_k$ and the nontrivial element of $G = \ints/2$ acts by sending
$z \to -z$, where $z$ is a coordinate.  
Taking $Y_R = \proj^1_R$ and having the nontrivial element of $G$ send
$z$ to $az$ clearly does not give a $G$-action on $Y_R$ if $a \in R$
is congruent to $-1$ modulo the maximal ideal but is not equal to $\pm
1$.   
\end{remark}

\begin{remark}
In fact, na\"{i}vely lifting as in Remark \ref{Rstupid} does not work,
\emph{even when $G$ is trivial}. For instance, take $Y \subseteq \proj^3_k$ to
be the twisted cubic given by the ideal $(xz - y^2, yw - z^2, xw -
yz)$.  If we lift these equations to characteristic zero, most choices
of coefficients will give a variety $Y_R$ where the generic fiber is
zero-dimensional, in which case $Y_R$ will not be flat over $R$!

The example above may seem silly, since it is easy to lift the
equations correctly (just keep all the coefficients $0$'s and $1$'s).
But in fact there do exist smooth varieties over $k$ that cannot be
lifted to characteristic zero.  The first example of this is a
threefold and is due to Serre (\cite{Se:ev}).
\end{remark}

\begin{remark}\label{Rcurvelifting}
A smooth projective \emph{curve} always lifts over any
\emph{complete} discrete valuation ring $R$ with residue field $k$ (\cite[III, Corollaire 6.10 and
Proposition 7.2]{SGA1}).  Thus Question \ref{Qgloballifting} for tame covers of
curves over $R$ has a positive answer: We first lift $X$ to a curve
$X_R$, we lift the branch points of $f$ to $R$-points of $X_R$ using
the valuative criterion for properness, and
then we apply Theorem \ref{Ttamelifting}.
\end{remark}

In spite of Remark \ref{Rcurvelifting}, the lifting problem does not always have a solution (see Examples
\ref{Ezpzpnonlift}, \ref{Eroquette}, and \ref{Efrobenius}).  However,
one of the major open conjectures in the field, giving a positive
lifting result, has been
recently solved by Obus-Wewers and Pop (\cite{OW:ce}, \cite{Po:oc}).

\begin{theorem}{\rm \textbf{(The Oort conjecture)}} 
The lifting problem for \emph{cyclic} covers of curves has a
solution (for some discrete valuation ring).
\end{theorem}

In \S\ref{Sglobal}, we discuss a number of basic examples of the
lifting problem.  It turns out that determining whether the lifting problem
can be solved reduces to a \emph{local lifting problem}
on extensions of power series rings (Question \ref{Qlocallifting}).
This reduction is thanks to a \emph{local-global principle}, which we
state and prove in \S\ref{Slocalglobal}.  

The local lifting problem is the main approach to the lifting problem today.  In \S\ref{Sgeneralities}, we define the local
lifting problem, and give some obstructions to solving it.  In
\S\ref{Ssummary}, we summarize the positive results for the
local lifting problem known at the present, which is followed by
\S\ref{Stechniques}, where we give an overview of the
various techniques that are used to construct lifts.  Although,
historically, positive results for the local lifting problem were
discovered before the discovery of systematic obstructions, we feel that it may be
beneficial for the reader to first have a limited set of cases of the
local lifting problem to think about, before seeing what is known.
The interested reader can certainly skip to \S\ref{Ssummary} right
after reading the introduction to \S\ref{Sgeneralities}.

In \S\ref{Sdeformations}, we sketch the deformation-theoretic approach
to the local lifting problem, which aims not only to solve it, but to
understand what the space of solutions looks like (this section
assumes familiarity with general deformation theory, and is written at
a somewhat higher level of sophistication than the rest of the chapter).  We close with
\S\ref{Sopen}, which is a list of open problems.  Appendix
\ref{Salgebraic} includes some algebraic results that are used in
our proofs.

\subsection{A remark on this exposition}
The paper \cite{Ob:ll} is an earlier exposition that I wrote on
the local lifting problem.  In what follows, I have attempted to minimize
duplication of \cite{Ob:ll} and to cite it for proofs whenever
possible.  At the same time, I have striven to keep notation as
identical as possible to that of \cite{Ob:ll}.  
This chapter contains a great deal of material not
present in \cite{Ob:ll}, such as a proof of the
local-global principle, the (differential) Hurwitz tree obstruction, the ``Mumford method'' of using
equicharacteristic deformations to build lifts and its application in
the resolution of the Oort conjecture, all the deformation-theoretic
material, several results of the last few years, and many examples and
open problems.  There is also quite a bit of material in \cite{Ob:ll}
not included here, in particular a detailed account of lifting for
cyclic extensions.  That being said, this
chapter has the same basic structure as \cite{Ob:ll}, and to make it
readable on its own, some repetition is necessary. In some sense, 
this chapter and \cite{Ob:ll} are companion papers that
can be read together if desired. In particular, I will often give a
reference to \cite{Ob:ll}, even when a result is originally from
another source (which I will also cite).

In order to keep the level of exposition relatively basic, we do not
include details on differential/deformation data, although they
are mentioned in \S\ref{Shurwitzobstruction}.

\subsection{Notation and conventions}\label{Snotation} 
Throughout this paper, $p$ represents a (fixed) prime number, $k$ is (unless
otherwise noted) an algebraically closed field of characteristic $p$, and $W(k)$ is
the ring of Witt vectors of $k$, that is, the unique complete discrete valuation ring in characteristic zero with uniformizer $p$
whose residue field is $k$ (see, e.g., \cite[II, \S6]{Se:lf}).  

If $\Gamma$ is a group of automorphisms of a ring $A$, we write $A^{\Gamma}$ for the fixed ring under $\Gamma$.
For a finite group $G$, a \emph{$G$-Galois extension} (or \emph{$G$-extension}) of rings is a finite
extension $A \hookrightarrow B$ (also written $B/A$) of integrally closed integral domains such that the associated extension of fraction fields is 
$G$-Galois.  We do \emph{not} require $B/A$ to be \'{e}tale.
 
If $x$ is a scheme-theoretic point of a scheme $X$, then $\mc{O}_{X,x}$ is the local ring of $x$ in $X$.  
If $R$ is any local ring, then $\hat R$ is the completion of $R$ with respect to its maximal ideal. 
A \emph{$G$-Galois cover} (or \emph{$G$-cover}) is a branched cover
$f: Y \to X$ with an isomorphism $G \cong \Aut(Y/X)$ such that $G$ acts transitively on 
each geometric fiber of $f$.  Note that $G$-covers of affine schemes give rise to $G$-extensions of rings, and vice versa.

Suppose $f: Y \to X$ is a branched cover, with $X$ and $Y$ locally noetherian.  If $x \in X$ and $y \in Y$ are smooth codimension $1$ points such that 
$f(y) = x$, then the \emph{ramification index} of $y$ is the ramification index of the extension of complete local rings 
$\hat{\mc{O}}_{X, x} \to \hat{\mc{O}}_{Y, y}$.  If $f$ is Galois, then
the ramification index of a smooth codimension $1$ point $x \in X$ 
is the ramification index of any point $y$ in the fiber of $f$ over $x$.  If $x \in X$ (resp.\ $y \in Y$) has 
ramification index greater than 1, then it is called a \emph{branch point} (resp.\ \emph{ramification point}).

If $R$ is any ring with a non-archimedean absolute value $| \cdot |$, then $R\{T\}$ 
is the ring of power series $\sum_{i=0}^{\infty} c_i T^i$ such that
$\lim_{i \to \infty} |c_i| = 0$.  Throughout the paper, we normalize 
the valuation on $R$ and $K$ so that $p$ has valuation $1$.

If $X$ is a smooth curve over a complete discrete valuation field $K$ with valuation ring $R$, then a \emph{semistable}
model for $X$ is a relative curve $X_R \to \Spec R$ with $X_R \times_R K \cong X$ and semistable special fiber (i.e.,
the special fiber is reduced with only ordinary double points for singularities).  

If $R$ is a discrete valuation ring with residue field $k$ and
fraction field $K$, and $A$ is an $R$-algebra, we write $A_k$ and
$A_K$ for $A \otimes_R k$ and $A \otimes_R K$, respectively. 

Suppose $S$ is a ring of characteristic zero, with an ideal $I$ such that $S/I$ has characteristic $p$.  If an indeterminate in $S$ is given by a capital
letter, our convention (which we will no longer state explicitly) will be to write its reduction in $S/I$ using the respective lowercase letter.  
For example, if $I \subseteq W(k)[[U]]$ is the ideal generated by $p$, then $W(k)[[U]]/I \cong k[[u]]$, and $u$ is the reduction of $U$. 

The group $D_n$ is the dihedral group of order $2n$.  The symbol
$\zeta_p$ represents a primitive $p$th root of unity.  The genus of a
curve $X$ is written $g(X)$.

\section{Global results}\label{Sglobal}
In this section, $f: Y \to X$ is a branched $G$-Galois cover of smooth
projective curves over $k$, where $G$ is a finite group.  Recall that
the lifting problem asks if there is a characteristic zero discrete valuation ring $R$
with residue field $k$ and a branched $G$-Galois cover $f_R: Y_R \to X_R$ of
smooth relative $R$-curves whose special fiber is $f$.  

\subsection{Tame covers}\label{Stame}
As we have seen in Remark \ref{Rcurvelifting} (based on Theorem
\ref{Ttamelifting}), if $R$ is \emph{complete} and $f$ is furthermore
tamely ramified, then $f$ lifts over $R$.  In particular, $f$ lifts to characteristic zero whenever $|G|$ is not
divisible by $p$.  It is not always easy to write down tame covers and their lifts
explicitly, but one can do this for \emph{cyclic} (or abelian) covers, as shown in
the next example.
 
\begin{example}\label{Eprimetop}
Suppose $f: Y \to \proj^1_k$ is a cyclic $\ints/n$-cover over $k$ corresponding to the
function field embedding 
$$k(t) \to k(t)[z]/(z^n - \prod_{i=1}^r (t - a_i)^{c_i})$$ for some
pairwise distinct $a_i \in k$, where $p \nmid n$ and all $c_i > 0$.  Let $\proj^1_R$ be the standard lift
of $\proj^1_k$ with coordinate $T$ reducing to $t$.  Let $R$ be a
complete characteristic zero discrete valuation ring with residue field $k$.  We claim
that a lift $f_R$ of $f$ can be given by taking the normalization $Y_R$ of $\proj^1_R$ in the function field
$$\mc{K} := K(T)[Z]/(Z^n - \prod_{i=1}^r (T - A_i)^{c_i}),$$ where $A_i$ is any lift
of $a_i$ to $R$.  Here $K = \Frac(R)$ and a generator of the Galois group takes $Z$
to $\zeta_nZ$ (note that $R$, being complete, contains $\zeta_n$).

The map $f_R$ is flat by Proposition \ref{Pflat}.  Since the normalization of $\aff^1_R$ in $\mc{K}$ contains $\Spec R[T,
Z]/(Z^n - \prod_{i=1}^r (T - A_i)^{c_i})$, it is clear that $Y_R
\times_R k$ is \emph{birationally equivalent} to $Y$.  But we must check that it is smooth (the normal
scheme $Y_R$ can, in theory, have singularities in codimension 2!).  The generic
fiber $f_K := f_R \times_R K$ is branched of index $e_i := n/\gcd(n, c_i)$
above $A_i$, so the Riemann-Hurwitz formula gives that the genus $g_{\eta}$ of $Y_R \times_R K$ satisfies
$$2g_{\eta} - 2 = -2n + \sum_{i=1}^r (n - e_i).$$  The same is true
for the genus $g_Y$ of $Y$.  Since $f_R$ is flat, the arithmetic genus $p_a(Y_R
\times_R k) = g_{\eta} = g_Y$ (e.g., \cite[III, Theorem  9.13]{Ha:ag}).  Thus $Y_R \times_R k$ is smooth (see,
e.g., \cite[IV, Exercise 1.8]{Ha:ag}).
\end{example}

\subsection{Wild covers}\label{Swild}
In stark contrast to the case of tame covers, the lifting problem for
wild covers contains a great deal of mystery.  Even in the most basic
example when a wild cover lifts, writing the lift down is less
straightforward than in Example \ref{Eprimetop}.

\begin{example}\label{Ezplift}
Let $f: Y = \proj^1 \to X = \proj^1$ be the Artin-Schreier $\ints/p$-cover over
$k$ given by the equation $z \mapsto z^p - z$, where $z$ is a coordinate
on $\proj^1$.  The Galois action is generated by $z \mapsto z+1$.  Let
$R$ be a complete discrete valuation ring with residue field $k$ containing $\zeta_p$ and let
$\lambda = \zeta_p - 1$.  A lift of $f$ to
characteristic zero is $f_R: Y_R = \proj^1_R \to X_R =
\proj^1_R$, where $f_R$ is given by 
$$Z \to \frac{(1 + \lambda Z)^p - 1}{\lambda^p}$$  
(note that this is defined over $R$, and reduces to $z \mapsto z^p -
z$ over $k$).  The $\ints/p$-Galois action on $Y_R$ is
generated by $Z \mapsto \zeta_pZ + 1$ (which has order $p$ --- the
reader should check this!), and this action reduces to $z
\mapsto z+1$ over $k$.  
\end{example}

\begin{remark}\label{Rbranchcollide}
Notice that the generic fiber of $Y_R$ above has \emph{two}
ramification points (at $Z = -1/\lambda$ and $Z = \infty$), both of
which specialize to the unique ramification point $z = \infty$ of
$Y$.  This phemomenon of a ramification point ``splitting'' into
several ramification points on a lift to characteristic zero happens whenever the
ramification point is wild (indeed, the Riemann-Hurwitz formula
necessitates this, as a wild ramification point of index $e$
contributes \emph{more} than $e-1$ to the degree of the ramification divisor).
\end{remark}

There are also examples of wild covers that do not lift to
characteristic zero, such as the following.

\begin{example}\label{Ezpzpnonlift}
Let $Y \cong \proj^1_k$. The group $G = (\ints/p)^n$ (for any $n$)
embeds into the additive group of $k$, and acts on $Y$ by the additive
action, fixing $\infty$.  Let $f: Y \to X \cong \proj^1$ be the
induced $G$-cover.   If $f_R: Y_R \to X_R$ is a lift of $f$ to
characteristic zero and $K = \Frac(R)$, then flatness of $Y_R \to
\Spec R$ implies that the generic fiber $Y_K$ of $Y_R$ is a genus zero
curve in characteristic zero with $G$-action.  However, the automorphism
group of $Y_K$ embeds into $PGL_2(\ol{K})$, which does not contain 
$(\ints/p)^n$ if $n > 1$ and $p^n \neq 4$.  So the $G$-action on $Y$
cannot lift to characteristic zero in these cases.
\end{example}

Along the same lines, and more simply, a cover
might not lift simply because the automorphism group is too large
for characteristic zero.  The following example is from \cite[\S4]{Ro:aa}.

\begin{example}\label{Eroquette}
Consider the smooth projective model $Y$ of the curve $y^2 = x^p - x$ over
$k$, where $p \geq 5$.  The genus $g_Y$ of $Y$ is $(p-1)/2$.  The
group $G := \Aut(Y)$ is generated by
\begin{eqnarray*}
&\sigma:& x \mapsto x+1, \quad y \mapsto y \\
&\tau:& x \mapsto ax \quad y \mapsto \sqrt{a}y \text{ ($a$ a generator
  of $\FF_p^{\times}$)}\\
&\upsilon:& x \mapsto -\frac{1}{x}, \quad  y \mapsto \frac{y}{x^{(p+1)/2}}
\end{eqnarray*}
This group contains $PGL_2(p)$ as an index $2$ subgroup (considering
only the action on $x$), so $|G| = 2p(p^2-1)$.  If the $G$-cover $f: Y \to
Y/\Aut(Y)$ lifted to characteristic zero, the generic fiber of the
lift would be a genus $g_Y$ curve with at least $|G|$ automorphisms.
Since $|G| > 84(g-1)$, this violates the Hurwitz bound on
automorphisms of curves in characteristic zero (see, e.g., \cite[IV,
Ex.\ 2.5]{Ha:ag}).  Indeed, this is the case even when $G$ is replaced
by its index 2 subgroup $PGL_2(p)$.
\end{example}

Lastly, here is an example from \cite[\S2]{Oo:la} of a wild cover that is non-liftable for reasons
other than the size of the automorphism group.

\begin{example}\label{Efrobenius}
Suppose $p = 5$, and let $G$ be the group of order $20$ with presentation $\langle \sigma,
\tau \,  | \, \sigma^5 =1, \tau^4 =1, \sigma\tau = \tau \sigma^{-1}
\rangle$.  
Let $f: Y \to X = \proj^1$ be the $G$-cover
corresponding to the embedding of function fields $k(t)
\hookrightarrow k(t)[x, y]/(x^4-t, y^5 - y - x^{-2})$, where the
$G$-action is given by 
$$\sigma^*y = y+1, \ \sigma^*x = x, \ \tau^*y = 4y, \ \tau^*x = 3x.$$  If $Z$
corresponds to the subfield $k(t, x)$, then $Z \to X$ is a $\ints/4$-cover of genus zero curves ramified at $x = 0$ and $x = \infty$, and
$Y \to Z$ is an Artin-Schreier $\ints/5$-cover branched only at $x =
0$.  If $P \in Y$ is the point above $x = 0$, then $xy^2$ is a
uniformizing parameter at $P$, and the ramification divisor $D$ of $Y \to
Z$ can be calculated by taking 
$$\frac{dx}{d(xy^2)} = -\frac{dx}{2xy\, dy} = -x^2/y,$$ whose divisor
has $P$-part $12[P]$.  Thus $D = 12[P]$, and the Riemann-Hurwitz
formula shows that the genus of $Y$ is $2$.

Now, suppose that $f$ has a lift $f_R: Y_R \to X_R$ over a discrete valuation ring $R$ in
characteristic zero.  There is an intermediate $\ints/5$-cover $Y_R
\to Z_R$ lifting $Y \to Z$.  By flatness, the generic fiber $Y_K \to
Z_K$ over $K := \Frac(R)$ is a $\ints/5$-cover of a genus $0$ curve by
a genus $2$ curve, and $Z_K \to X_K$ is a $\ints/4$-cover of genus
zero curves.  Since all ramification points of $Y_K \to Z_K$
have ramfication index $5$, the Riemann-Hurwitz formula yields 
$$2(2) - 2 = 5(-2) + 4r,$$ where $r$ is the number of branch 
points.  Thus $r = 3$, and these three points are permuted by the
$\ints/4$-action on $Z_K$.  
Since $Z_K \to X_K$ is a cover of genus zero curves, this action is
free apart from two fixed points.  So no set of three points is stable
under this action, yielding a contradiction.
\end{example}

The examples above motivate the following obstruction to lifting,
known as the \emph{Katz-Gabber-Bertin} (or \emph{KGB}) obstruction
(cf.\ \cite[\S1]{CGH:ll}, where the definition is given in a slightly
different context).

\begin{definition}{\textbf{(The KGB Obstruction)}}\label{DglobalKGB}
If $f: Y \to X$ is a branched $G$-cover of smooth projective curves
over $k$, then there is \emph{KGB obstruction} to lifting $f$ if there exists
no curve $C$ in characteristic zero with faithful $G$-action such that, for all $H \leq G$, the
genus of $C/H$ equals the genus of $Y/H$.  If there does exist such a
curve, we say that \emph{the KGB obstruction vanishes}.
\end{definition}

\begin{remark}
The KGB obstruction above motivates the \emph{local KGB obstruction}
(Definition \ref{DKGB}), which will be the form we primarily use. 
\end{remark}

By flatness, it is clear that if $f$ has a KGB obstruction to lifting,
then it does not lift to characteristic zero.  Furthermore, the KGB
obstruction explains the failure of lifting in Examples
\ref{Ezpzpnonlift}, \ref{Eroquette}, and \ref{Efrobenius}.  

\begin{remark}\label{Robstructions}
In general, obstructions to the lifting problem come from looking at
all conceivable lifts of characteristic $p$ branched covers with certain
properties, seeing what properties the lifts would have, and attempting to encode
these properties in an abstract structure.  There will then be an
obstruction to lifting if this abstract structure can't exist.  As we
have seen, the KGB obstruction rests on the simple observation that a
lift of a curve with a given genus has the same genus.  The
abstract structure in this case can be viewed as a function from subgroups $H \subseteq G$ to integers (representing the
genus of $Y/H$).  This function can exist only when it obeys the
constraints given by the Riemann-Hurwitz formula in characteristic zero. 
\end{remark}

\subsection{Oort groups and the Oort conjecture}\label{Soort}

Remark \ref{Rcurvelifting} has the consequence that every $G$-Galois
cover $f: Y \to X$ of smooth projective curves over $k$, where $p \nmid |G|$, lifts to
characteristic zero.  This motivates the natural question of which
groups $G$ have the property that all $G$-covers lift.  After \cite{CGH:og}, we call a finite
group $G$ with this property an \emph{Oort group for $p$} (we
sometimes suppress $p$ when it is implicit).  Thus prime-to-$p$
groups are Oort groups.  By the examples of \S\ref{Swild},
$(\ints/p)^n$ is not an Oort group when $n > 1$ and $p^n > 4$, the
group $PGL_2(p)$ is not an Oort group for $p > 3$, and the group of
order $20$ with presentation $\langle \sigma,
\tau \,  | \, \sigma^5 =1, \tau^4 =1, \sigma\tau = \tau \sigma^{-1}
\rangle$ is not an Oort group for $5$.  From this, one can easily obtain more
examples of non-Oort groups (for instance, it is an exercise
to show that any direct product of a non-Oort group with another group
will be a non-Oort group).  

On the other hand, it is generally \emph{not} easy to prove that a group is an Oort
group.   The first major result for a group with
order divisible by $p$ is due to Sekiguchi-Oort-Suwa.

\begin{theorem}[\cite{SOS:ask}]\label{Tzplifting}
If $G$ is a cyclic group of order $mp$, with $p \nmid m$, then $G$ is
an Oort group.
\end{theorem}


Around the same time, Oort (\cite[\S7]{Oo:la}) made the statement that 
``it seems reasonable to expect that [lifting] is possible for every
automorphism of an algebraic curve."  This is
equivalent to saying that all cyclic groups are Oort groups (for all
$p$).  This statement has come to be known as the \emph{Oort conjecture},
and has driven much of the progress to date in the lifting problem.
The Oort conjecture was recently proven by the combined work of
Obus-Wewers (\cite{OW:ce}) and Pop (\cite{Po:oc}).  The proof makes
heavy use of the \emph{local-global principle}, which is the subject
of the next section.  In fact, the local-global principle underlies virtually all progress on the lifting
problem since \cite{SOS:ask}.  Thus, we will discuss it before saying anything further
about Oort groups.

\section{The local-global principle}\label{Slocalglobal}
In this section, $R$ is a \emph{complete} characteristic zero discrete valuation ring with
residue field $k$.

Suppose a finite group $G$ acts on a smooth, projective curve $Y/k$,
giving rise to a $G$-cover $f: Y \to X := Y/G$.  
Let $y \in Y$ be a closed point, and let $I_y \leq G$ be the inertia group of $y$.  The
$G$-action on $Y$ induces an $I_y$-action on the complete local
ring $\hat{\mc{O}}_{Y, y}$, which is isomorphic to a power series ring
$k[[z]]$, since $Y$ is smooth (Lemma \ref{Lcohen}).  The fixed ring
$\hat{\mc{O}}_{Y,y}^{I_y}$ is the complete local ring $\hat{\mc{O}}_{X, f(y)}$, so is also a power series ring in one variable, say
$k[[t]]$.

Now, suppose further that $f: Y \to X := Y/G$ lifts to a
characteristic zero $G$-cover $f_R: Y_R \to X_R$.  We identify the
special fiber of $f_R$ with $f$.  Then each $y
\in Y$ as above gives rise to an extension of complete local rings $\hat{\mc{O}}_{Y_R,
  y}/\hat{\mc{O}}_{X_R, f(y)}$.  The group $I_y$ acts on
$\hat{\mc{O}}_{Y_R, y}$, this action lifts the action of $I_y$ on
$\hat{\mc{O}}_{Y, y}$, and $\hat{\mc{O}}_{X_R, f(y)} =
\hat{\mc{O}}_{Y_R, y}^{I_y}$.  In fact, there exist isomorphisms $\hat{\mc{O}}_{Y_R,
  y} \cong R[[Z]]$ and $\hat{\mc{O}}_{X_R, f(y)} \cong R[[T]]$, where
$Z$ and $T$ can be chosen to be any lifts of $z$ and $t$ (Lemma \ref{Llocalring}). Thus, it makes sense to say that the
$I_y$-extension $R[[Z]]/R[[T]]$ is a \emph{lift of the $I_y$-extension
$k[[z]]/k[[t]]$ to characteristic zero}.  In other words, {\bf any lift of
the $G$-cover $f$ over $R$ gives rise to a lift
of the $I_y$-extension $\hat{\mc{O}}_{Y, y}/\hat{\mc{O}}_{X, f(y)}$
over $R$}.

Amazingly, the converse of the above statement holds as well!
This is called the \emph{local-global principle} for lifting. 

\begin{theorem}{\rm \textbf{(Local-global principle)}}\label{Tlocalglobal}
Let $f: Y \to X$ be a $G$-cover of smooth, projective, $k$-curves.
For each closed point $y \in Y$, let $I_y \leq G$ be the inertia group.
If, for all $y$, the $I_y$-extension $\hat{\mc{O}}_{Y, y} / \hat{\mc{O}}_{X, f(y)}$ lifts over
$R$, then $f$ lifts over $R$.
\end{theorem}

\begin{remark}
If $y$ is \emph{unramified} in $f$ above, then $I_y$ is trivial, so the
extension $\hat{\mc{O}}_{Y, y} / \hat{\mc{O}}_{X, f(y)}$ lifts
automatically.  Thus, in light of Theorem \ref{Tlocalglobal}, one need
only check the (finitely many) ramified points $y \in Y$ in order to
show that $f$ lifts to characteristic zero.
\end{remark}

Our proof of the local-global principle uses \emph{formal patching} 
in the spirit of \cite[\S1.2]{Sa:fl} and \cite[\S3]{He:thesis} (see \cite{Ha:pg} for a detailed introduction to patching,
and \cite[II.9]{Ha:ag} for an introduction to formal schemes).  The idea
is as follows: Suppose we are given a branched Galois cover $f: Y \to X$ and lifts of
the Galois covers $\hat{\mc{O}}_{Y, y} / \hat{\mc{O}}_{X, f(y)}$ over
$R$ for all ramification points $y$.  Let $U \subseteq X$ and $V
\subseteq Y$ be the complements of the branch and ramification loci,
respectively.  The cover $f: V \to U$ is \'{e}tale, and thus admits a
lift to a formal scheme over $R$ due to Grothendieck's theory of \'{e}tale
morphisms.  The individual Galois covers $\hat{\mc{O}}_{Y, y} /
\hat{\mc{O}}_{X, f(y)}$ admit lifts over $R$ by assumption.  These can
be ``patched'' together to create a Galois branched cover of proper
formal schemes $\mc{F}: \mc{Y} \to \mc{X}$ over $R$.  Grothendieck's Existence
Theorem (also known as ``formal GAGA'') shows that $\mc{F}$ is
actually the formal completion of a $G$-Galois branched cover of $R$-curves $f_R:
Y_R \to X_R$, which is the lift we seek.

The rest of \S\ref{Slocalglobal} will be devoted to the details of
this proof.

\subsection{Preliminaries and \'{e}tale lifting}\label{Sformal}
We start with two basic lemmas on formal lifting.

\begin{lemma}\label{Lsinglecurveformallift}
Let $X$ be a smooth curve over $k$ and let $R$ be a complete discrete valuation ring with
residue field $k$.  There exists a smooth formal
$R$-curve $\mc{X}$ whose special fiber is $X$.  Furthermore, if $U
\subset X$ is an open subscheme, there is a smooth formal $R$-curve
$\mc{U} \subseteq \mc{X}$ whose special fiber is $U$.
\end{lemma}

\begin{proof}
Let $\ol{X}$ be a smooth projective completion of $X$.  By Remark
\ref{Rcurvelifting}, there is a lift of $\ol{X}$ to a smooth projective
curve $\ol{X}_R$ over $R$.  By Hensel's lemma, we can lift each of the
points of $\ol{X} \backslash X$ to an $R$-point of $\ol{X}_R$.  Let
$X_R$ be the affine $R$-curve given by the complement of these
$R$-points.  Now take $\mc{X}$ to be the formal completion of $X_R$ at
$X$.  This process suffices as well for the construction of $\mc{U}$.
\end{proof}

\begin{lemma}[{\cite[I, Corollaire 8.4]{SGA1}}]\label{Letalelifting}
Let $f: Y \to X$ be an \'{e}tale cover of smooth $k$-curves and let
$R$ be a complete discrete valuation ring with residue field $k$.  Let
$\mc{X}$ be a smooth formal $R$-curve with special fiber $X$.  Then
there is a unique smooth formal $R$-curve $\mc{Y}$ with a finite map $\mc{F}: \mc{Y}
\to \mc{X}$ such that $\mc{F}$ has special fiber $f$.  If $f$ is
$G$-Galois, then so is $\mc{F}$, and the $G$-action on $\mc{Y}$ lifts
that on $Y$.
\end{lemma}

\subsection{A patching result}\label{Spatching}

\begin{prop}[cf.\ {\cite[Proposition 4.1]{He:thesis}}]\label{Ppatching}
Let $R$ be a complete discrete valuation ring with residue field $k$ and uniformizer
$\pi$.  Let $\mc{X}$ be a smooth formal affine $R$-curve with special
fiber $X$.  Suppose $f: Y \to X$ is a finite, dominant, separable
morphism.  Let $x \in X$ be a closed point, let $U = X \backslash \{x\}$, let $V =
f^{-1}(U)$, and assume that $f|_V$ is \'{e}tale.  Let $\mc{U}
\subseteq \mc{X}$ be a formal lift of $U$ over $R$
as in Lemma \ref{Lsinglecurveformallift}, and let $\Phi: \mc{V} \to
\mc{U}$ be the lift of $f|_V: V \to U$ guaranteed by Lemma
\ref{Letalelifting}.  Lastly, let $A$ be a finite normal $\hat{\mc{O}}_{\mc{X},
  x}$-algebra such that there is an isomorphism $A/\pi \stackrel{\sim}{\to} \prod_{y
  \in f^{-1}(x)} \hat{\mc{O}}_{Y, y}$.

\begin{enumerate}[(i)]
\item There is a smooth formal $R$-curve $\mc{Y}$ containing $\mc{V}$
and a finite cover $\mc{F}: \mc{Y} \to \mc{X}$ such that
$\mc{F}|_{\mc{V}} = \Phi$, that the special fiber of $\mc{F}$ is $f$,
and that the restriction of $\mc{F}$ above $\Spec \hat{\mc{O}}_{\mc{X}, x}$
gives the extension $A / \hat{\mc{O}}_{\mc{X}, x}$.  

\item If $f$ is $G$-Galois, $G$ acts on $A$ with $A^G \cong
\hat{\mc{O}}_{\mc{X}, x}$, and the morphism $A \to \prod_{y \in
  f^{-1}(x)} \hat{\mc{O}}_{Y, y}$ is $G$-equivariant, then $\mc{F}$ is
$G$-Galois, and the $G$-Galois action on $\mc{Y}$ lifts that on $Y$.
\end{enumerate}
\end{prop}

\begin{proof}
We adapt the proof of Henrio.
Write $f^{-1}(x) = \{y_1, \ldots, y_n\}$.  Fix isomorphisms
$\hat{\mc{O}}_{Y, y_i} \cong k[[z_i]]$ and $\hat{\mc{O}}_{X, x} \cong
k[[t]]$. Then 
$$\hat{\mc{O}}_{Y, y_i} \cong k[[z_i]] \cong \hat{\mc{O}}_{X,
  x}[z_i]/p_i(z_i),$$
where $p_i$ is a separable monic Eisenstein polynomial with coefficients
in $\hat{\mc{O}}_{X, x} \cong k[[t]]$.   

Fix an isomorphism $\hat{\mc{O}}_{\mc{X}, x}
\cong R[[T]]$, with $T$ lifting $t$ (Lemma \ref{Llocalring}).  This fixes an isomorphism
$(\hat{\mc{O}}_{\mc{X}, x})^{\wedge}_{(\pi)} \cong R[[T]]\{T^{-1}\}$ of
complete discrete valuation rings.  Now, $B := A \otimes_{R[[T]]} R[[T]]\{T^{-1}\}$ is finite over
$R[[T]]\{T^{-1}\}$, and reducing modulo $\pi$ yields $B/\pi \cong \prod_{i=1}^n
k((z_i)) \cong \prod_{i=1}^n k((t_i))[z_i]/p_i (z_i),$ where $n = |f^{-1}(x)|$.  Let $P_i \in
R[[T]]\{T^{-1}\}[S]$ be a monic polynomial lifting $p_i$.  By
Hensel's lemma, $B$ is a product of finite $R[[T]]\{T^{-1}\}$-algebras $B_i$
($1 \leq i \leq n$) where $B_i$ contains a root $Z_i$ of $P_i$ lifting
$z_i$ and $B_i/\pi \cong k((z_i))$ as a $k((t))$-algebra.  One obtains homomorphisms
$$u_i: R[[T]]\{T^{-1}\}[Z_i]/P_i(Z_i) \to B_i.$$  This homomorphism is an
isomorphism modulo $\pi$, and is thus an isomorphism by Remark
\ref{Risom}. Note that, under the assumptions in (ii), $G$ acts on $B$ by
$R[[T]]\{T^{-1}\}$-homomorphisms.

On the other hand, let $C$ be the finite $R[[T]]\{T^{-1}\}$-algebra
given by $\mc{O}(\mc{V}) \otimes_{\mc{O}(\mc{U})} R[[T]]\{T^{-1}\}$, where 
$\mc{O}(\mc{U}) \subseteq R[[T]]\{T^{-1}\}$ is the natural inclusion.
Then $$C/\pi \cong \mc{O}(V) \otimes_{\mc{O}(U)} k((t)) \cong
\prod_{i=1}^n k((z_i)).$$ Proceeding as with $B$, we can write $C =
\prod_{i=1}^n C_i$, with isomorphisms
$$w_i: R[[T]]\{T^{-1}\}[Z_i]/P_i(Z_i) \to C_i.$$ Note that, under the
assumptions in (ii), Lemma \ref{Letalelifting} shows that $\mc{V} \to
\mc{U}$ is $G$-Galois, and thus that $G$ acts on $C$ by
$R[[T]]\{T^{-1}\}$-homomorphisms.

The maps $\mu_i := u_i \circ w_i^{-1}: C_i \to B_i$ are
$R[[T]]\{T^{-1}\}$-isomorphisms, giving rise to a product isomorphism
$\mu: C \to B$.  Then, $\mu$ modulo $\pi$ is the identity
on $\prod_{i=1}^n k((z_i))$.  Thus $\mu$ is $G$-equivariant, since roots of $P_i$ are
in one-to-one correspondence with roots of $p_i$.  

We can define an $\mc{O}(\mc{X})$-module homomorphism
$\theta: \mc{O}(\mc{V}) \times A \to B$ by 
$$\theta(h, g) = \mu(h' \otimes 1) - g \otimes 1.$$
Since one can find a rational function on an affine curve (namely, $V$) with
specified principal part at finitely many points (exercise!), the map
$\theta$ is surjective modulo $\pi$.  By Remark \ref{Risom},
$\theta$ is surjective.  Under the assumptions in (ii), $\theta$ is
clearly $G$-equivariant.

Let $\mc{A} = \ker(\theta)$, which is an
$\mc{O}(\mc{X})$-algebra.  We claim that we can take $\mc{Y} = \Spf
\mc{A}$.  Since $\mc{A}/\pi \cong \mc{O}(Y)$ is a finite
$\mc{O}(X)$-algebra, $\mc{A}$ is a finite $\mc{O}(\mc{X})$-algebra
(this follows from Lemma \ref{Lhomological} applied to a lift of a
presentation for $\mc{A}/\pi$).  Thus we obtain a finite morphism
$\mc{F}: \mc{Y} \to \mc{X}$.  Furthermore, we have $\mc{A}
\otimes_{\mc{O}(\mc{X})} \mc{O}(\mc{U}) \cong \mc{O}(\mc{V})$, using
Remark \ref{Risom} and the fact that the isomorphism holds
modulo $\pi$ (where it simply expresses the fact that $Y \otimes_X U
\cong V$).  Geometrically, this means that
$\mc{F}|_{\mc{V}} = \Phi$.  Similarly, $\mc{A} \otimes_{\mc{O}(\mc{X})}
R[[T]] \cong A$ (again, looking modulo $\pi$, this just says that $Y
\times_X \Spec \hat{\mc{O}}_{X, x} = \prod_{i=1}^n \hat{\mc{O}}_{Y,
  y_i}$).  The claim is proved, proving (i).
Under the assumptions of (ii), since $\mu$ is
$G$-equivariant, $G$ acts on $\mc{A}$ and the action restricts to that
on $\mc{O}(Y)$.  This proves (ii).
\end{proof}

\subsection{Proof of Theorem \ref{Tlocalglobal}}\label{Sprooflocalglobal}
Let $x_1, \ldots, x_n$ be a finite set of closed points of $x$ containing the branch locus of
$f$, with $n \geq 2$.  Let $U = X \backslash \{x_1, \ldots, x_n\}$,
let $V = f^{-1}(U)$, let $X_i = U \cup \{x_i\} \subseteq X$, and let $Y_i
= f^{-1}(X_i)$.  Note that $X_i$ and $Y_i$ are affine since $n \geq
2$.  Let $X_R$ be a smooth lift of $X$ over $R$, and let $\mc{X}$ be
the formal completion of $X_R$ at the uniformizer $\pi$ of $R$.  Let $\mc{X}_i \subseteq \mc{X}$ (resp.\ $\mc{U}
\subseteq\mc{X}$) be the lift of $X_i$ (resp.\ $U$)
guaranteed by Lemma \ref{Lsinglecurveformallift}, and let $\Phi: 
\mc{V} \to \mc{U}$ be the $G$-cover of formal curves guaranteed by
Lemma \ref{Letalelifting}.  Lastly, let $A_i$ be the $R[[T]]$-algebra $\Ind_{I_{y_i}}^G R[[Z]]$,
where $R[[Z]]/R[[T]]$ is a lift of the $I_{y_i}$-extension
$\hat{\mc{O}}_{Y, y_i}/\hat{\mc{O}}_{X, x_i}$ for some point $y_i$
lying above $x_i$.  

We apply Proposition \ref{Ppatching} with $f|_{Y_i}: Y_i \to X_i$ in
the role of $f: V \to U$ and with $A_i$ in the
role of $A$.  The assumptions of Proposition \ref{Ppatching}(ii) are
satisfied, and we obtain $G$-Galois covers $\mc{F}_i: \mc{Y}_i \to
\mc{X}_i$, all of which restrict to $\Phi$ on $V$.  Now, a $G$-Galois
cover of $\mc{X}_i$ corresponds to a coherent sheaf $\mc{B}_i$ of
$G$-Galois $\mc{O}_{\mc{X}_i}$-algebras.  Using (Zariski!)
gluing, the sheaves $\mc{B}_i$ glue to a coherent sheaf
$\mc{B}$ of $\mc{O}_{\mc{X}}$-algebras on the proper formal scheme $\mc{X}$.  By Grothendieck's Existence
Theorem (\cite[Th\'{e}or\`{e}me 3]{gfga}), $\mc{B}$ is the $\pi$-adic completion
of a coherent sheaf $B_R$ of $G$-Galois $\mc{O}_{X_R}$-algebras
(technically, \cite{gfga} only gives the result for sheaves of \emph{modules},
but since Grothendieck's Existence Theorem in fact gives an equivalence of \emph{tensor} categories of
coherent sheaves on $\mc{X}$ and $X_R$, one gets the same result for
sheaves of \emph{algebras}, as well as for sheaves of $G$-Galois algebras.  See
\cite[General Principle 2.2.4]{Ha:pg}  and the surrounding discussion
for details.  The sheaf $B_R$ gives rise to a $G$-cover $f_R: Y_R \to X_R$ of proper,
smooth $R$-curves whose special fiber is $f$.  This is the lift we seek. \qed

\begin{remark}\label{Rbertinmezard}
Bertin and M\'{e}zard gave an alternate proof of Theorem
\ref{Tlocalglobal} that relies on deformation theory (see
\S\ref{Sdeflocalglobal} below and \cite[\S3]{BM:df}).
There is also a proof by Green and Matignon (\cite[III]{GM:lg}).
\end{remark}


\section{The local lifting problem and its obstructions}\label{Sgeneralities}
In this section, $m$ represents a natural number prime
to $p$.  In light of the local-global principle (Theorem
\ref{Tlocalglobal}), the following question, known as the \emph{local
  lifting problem}, is the key to solving the lifting problem.

\begin{question}\label{Qlocallifting}
Suppose $G$ is a finite group, and $k[[z]]/k[[t]]$ is a $G$-Galois
extension.  Does there exist a
characteristic zero discrete valuation ring
$R$ with residue field $k$ and a $G$-Galois extension $R[[Z]]/R[[T]]$ such that the $G$-action on $R[[Z]]$ reduces to the given $G$-action on $k[[z]]$?
\end{question}
If such a lift exists, we say that $R[[Z]]/R[[T]]$ is a \emph{lift} of $k[[z]]/k[[t]]$ over $R$, or that $k[[Z]]/k[[t]]$ \emph{lifts to characteristic zero} (or
\emph{lifts over $R$}).\\

\begin{remark}\label{Rsolvable}
As one sees in \S\ref{Sramification}, any group $G$ for which
there exists a faithful local $G$-extension is of the form $P \rtimes
\ints/m$, where $P$ is a $p$-group.  We will call such a group
\emph{cyclic-by-$p$}.  
In particular, cyclic-by-$p$ groups are \emph{solvable}.  This is
one of the benefits of the local-global principle --- it converts a
problem that potentially deals with all finite groups into a problem
dealing only with solvable groups.
\end{remark}

An $G$-extension $k[[z]]/k[[t]]$ will be known as a \emph{local
  $G$-extension}.  Since the lifting problem does not always have a solution, neither does the local lifting problem.  So we ask:
 \emph{Can we find necessary and sufficient criteria for a local $G$-extension to lift to characteristic zero?  Over some particular $R$?} 

A cyclic-by-$p$ group for which every local $G$-extension lifts to characteristic zero
is called a \emph{local Oort group} for $k$.  If there exists a local
$G$-extension lifting to characteristic zero, $G$ is called a
\emph{weak local Oort group}.  The question of whether a group $G$ is
a weak local Oort group has been called ``the inverse Galois problem''
for the local lifting problem (\cite{Ma:pg}).

\begin{example}\label{Elocaltame}
The group $\ints/m$ is a local Oort group for all $p \nmid m$.  This is because, up to a
change of variable, every local $\ints/m$-extension is $k[[z]]/k[[t]]$
where $z^m = t$ and the Galois action is generated by $z \mapsto
\zeta_m z$.  The lift is simply $R[[Z]]/R[[T]]$ with Galois action
generated by $Z \mapsto \zeta_mZ$.
\end{example}

\begin{example}\label{Enegativezpzp}
By the local-global principle, every $G$-cover that does not lift to
characteristic zero yields a local $H$-extension that does not lift to
characteristic zero, where $H$ is the inertia group of some
ramification point.  For instance, let $\iota: G = (\ints/p)^n \to k^+$ be
an embedding. Example \ref{Ezpzpnonlift} gives
rise to the $(\ints/p)^n$-action on $k[[u]]$ such that $\sigma \in G$
sends $u$ to $u/(1 + \iota(\sigma) u)$ (here, $u = 1/z$, where $z$ is
the coordinate on $\proj^1$).  Let $k[[t]] = k[[u]]^G$.  Then, for $n
> 1$ and $p^n > 4$, the local $G$-extension $k[[u]]/k[[t]]$ does not
lift to characteristic zero.  In particular, $(\ints/p)^n$ is not a
local Oort group (although it is a weak local Oort group, by
\cite{Ma:pg}).
\end{example}

\subsection{The (local) KGB obstruction}\label{SlocalKGB}
The KGB obstruction (Definition \ref{DglobalKGB}) has a local
counterpart called the \emph{local KGB obstruction} (often, one
drops the word ``local,'' which should not cause much confusion).  This
is the main tool used to show that local $G$-extensions do not lift to
characteristic zero.

Let $k[[z]]/k[[t]]$ be a local $G$-extension, where $G \cong P \rtimes
\ints/m$, with $P$ a $p$-group and $p \nmid m$.  A theorem of Harbater (\cite{Ha:mp})
in the case $m=1$ and Katz and Gabber
(\cite[Theorem 1.4.1]{Ka:lg}) in the general case
states that there exists a unique $G$-cover $Y \to \proj^1_k$ that is \'{e}tale outside $t \in \{0, \infty\}$, tamely ramified of 
index $m$ above $t = \infty$, and totally ramified above $t = 0$ such
that the $G$-extension of complete local rings at $t=0$ is given by $k[[z]]/k[[t]]$.  This is called the 
\emph{Harbater-Katz-Gabber (HKG) cover associated to $k[[z]]/k[[t]]$}.  By Theorem \ref{Tlocalglobal} and
Example \ref{Elocaltame}, the $G$-cover $f: Y \to \proj^1_k$ lifts to
characteristic zero if and only if the extension $k[[z]]/k[[t]]$ does.  

\begin{definition}\label{DKGB}
In the context above, we say that a local $G$-extension has a
\emph{(local) KGB obstruction to lifting} if the associated HKG-cover does.
\end{definition} 

\begin{remark}
One can also formulate the local KGB obstruction in a purely local
way (not using HKG-covers), using the different and Proposition
\ref{Pdifferent} below as a replacement for the fact that the genus of
a lift of a curve equals the genus of the original curve.  We leave
this as an exercise.   
\end{remark} 

\begin{remark}
The \emph{Bertin obstruction} of \cite{Be:ol} is strictly weaker than
the local KGB obstruction, so we
do not discuss it further.  That the local KGB obstruction is at least
as strict as the Bertin obstruction is proven in
\cite[Theorem 4.2]{CGH:ll}.  An example of a local
$\ints/3 \times \ints/3$-extension with vanishing Bertin obstruction
but non-vanishing local KGB obstruction is given in \cite[Example B.2]{CGH:ll}. 
\end{remark}

Clearly, if $k[[z]]/k[[t]]$ lifts to characteristic zero, its local KGB
obstruction vanishes.  If $G$ is cyclic-by-$p$ and
the local KGB obstruction vanishes for \emph{all} local
$G$-extensions, then $G$ is called a \emph{local KGB group} for $p$.
If the local KGB obstruction vanishes for \emph{some} local
$G$-extension, then $G$ is called an \emph{weak local KGB group} for $p$.

The following classification of the local KGB groups is due to Chinburg, Guralnick, and Harbater.

\begin{theorem}[\cite{CGH:ll}, Theorem 1.2]\label{TKGB}
The local KGB groups for $k$ consist of the cyclic groups, the dihedral group $D_{p^n}$ for any $n$, the group $A_4$ (for $\cf(k) = 2$),
and the generalized quaternion groups $Q_{2^m}$ of order $2^m$ for $m \geq 4$ (for $\cf(k) = 2$).
\end{theorem}

\begin{proof}[Sketch of proof.]
We briefly
outline the \emph{negative} direction
(i.e., that there are no local KGB groups aside from the ones on the list).  
The first observation is that if $G$ is a local KGB group for $p$, then any
quotient of $G$ is as well.  This is because any local $G/H$-extension
can be extended to a local $G$-extension (\cite[Lemma 2.10]{CGH:og}), and if the local
$G/H$-extension has nontrivial local KGB obstruction, then the $G$-extension
clearly has one too.  Thus, to show that a group is not local KGB,
it suffices to show it has a quotient that is not local KGB.  There is an
explicit list of types of groups that can be shown not to be local KGB (this
list includes, for example, $\ints/p \times \ints/p$ for $p$ odd), and
it can be further shown that any cyclic-by-$p$-group either has a quotient on this list, or is one of the groups in Theorem
\ref{TKGB} (see \cite[\S11-12]{CGH:ll}).  This completes the proof. 
\end{proof}

For examples illustrating the KGB obstruction for $\ints/p^n \rtimes
\ints/m$ and additional examples for $\ints/p \times
\ints/p$, see \cite[Propositions 5.8,
5.9]{Ob:ll}.

Since any local Oort group is a local KGB group, the search for local
Oort groups is restricted to the groups in Theorem \ref{TKGB}.  
The generalized quaternion groups were shown \emph{not} to be local Oort groups in \cite{BW:ac}.  The obstruction developed in \cite{BW:ac} 
to show this is called the \emph{Hurwitz tree obstruction}, and will
be discussed in \S\ref{Shurwitzobstruction} (see specifically Example \ref{Ebrewiswewers}).  

\subsubsection{Global consequences}\label{SKGBconsequences}

Suppose $G$ is an \emph{arbitrary} finite group and fix a prime $p$.  Let us revisit the
question of whether $G$ is an Oort group (for $p$).  By
the local-global principle, it is clear that if every cyclic-by-$p$
subgroup of $G$ is local Oort, then $G$ is an Oort group.  In fact,
the converse is true as well.

\begin{prop}[{\cite[Theorem 2.4]{CGH:og}}]\label{Plocalglobaloort}
If $G$ is a finite group, then it is an Oort group for $p$ if and only if
every cyclic-by-$p$ subgroup of $G$ is a local Oort group.
\end{prop}

\begin{proof}
The ``if'' direction follows from the local-global principle.  To
prove the ``only if'' direction, let $I \leq G$ be a cyclic-by-$p$
subgroup, and let $k[[z]]/k[[t]]$ be a local $I$-extension.  By
\cite[Lemma 2.5]{CGH:og}, there is a $G$-Galois cover $f: Y \to
\proj^1$ such that there is a point $x \in \proj^1$ and a point $y \in
Y$ above $X$ for which the
$I$-Galois extension $\hat{\mc{O}}_{Y, y} / \hat{\mc{O}}_{\proj^1, x}$
is isomorphic (as an $I$-extension) to $k[[z]]/k[[t]]$.  Since $G$ is
an Oort group, $f$ lifts to characteristic zero.  By the easy
direction of the local-global principle, $k[[z]]/k[[t]]$ lifts to
characteristic zero as well.  So $I$ is a local Oort group.
\end{proof}

The following example is contained in \cite[Corollary 1.4]{Ob:A4}. 
\begin{example}\label{EA5oort}
The group $A_5$ is an Oort group for every prime.  Indeed, its only
cyclic-by-$p$ groups (for any $p$) are $\ints/2$, $\ints/3$,
$\ints/5$, $\ints/2 \times \ints/2$, $S_3$, $A_4$, and $D_5$, all of
which are local Oort groups for their respective primes (see \S\ref{Ssummary}).
\end{example}

In light of Proposition \ref{Plocalglobaloort}, classifying the 
Oort groups for $p$ consists of two parts: classifying the local
Oort groups for $p$, and classifying the groups whose cyclic-by-$p$
subgroups are on this list.  Following
Chinburg-Guralnick-Harbater (\cite{CGH:go}), we call a group $G$ an
\emph{O-group} for $p$ if every cyclic-by-$p$ subgroup of $G$ is either
cyclic, $D_{p^n}$, or $A_4$ if $p = 2$.  Since these are the only
cyclic-by-$p$ groups that can be local Oort groups, Proposition
\ref{Plocalglobaloort} shows that a group $G$ can be an Oort group only if
it is an O-group.  In particular, if $D_{p^n}$ is shown to be
local Oort for all $p$ and all $n$, then the list of O-groups is the
same as the list of Oort groups.  The list of all O-groups has been
computed by Chinburg-Guralnick-Harbater.

\begin{theorem}[{\cite[Theorems 2.4, 2.6]{CGH:go}}]\label{TOgroups}
If $p$ is an \emph{odd} prime, then $G$ is an O-group for $p$ if and
only if a $p$-Sylow subgroup $P$ of $G$ is cyclic and either
\begin{itemize}
\item The normalizer $N_G(P)$ and centralizer $Z_G(P)$ of $P$ in $G$
  are equal, or,
\item $|N_G(P)/Z_G(P)| = 2$, the order $p$ subgroup $Q \leq P$ has
  abelian centralizer, and every element of $N_G(Q) \backslash Z_G(Q)$
  acts as an involution inverting $Z_G(Q)$.
\end{itemize}

If $p = 2$, then $G$ is an O-group if and only if $P$ is cyclic or $P$
is dihedral, with $Z_G(K) = K$ for all elementary abelian subgroups $K$ of
order $4$.
\end{theorem}
More explicit lists of groups satisfying these criteria are given in
\cite[Theorems 2.7, 3.8]{CGH:go}.

\subsubsection{Consequences of obstructions for weak local Oort groups}\label{Sweak}

Suppose $G := \ints/p^n \rtimes \ints/m$ is non-cyclic (equivalently,
non-abelian).  It turns out that a local $G$-extension
$k[[z]]/k[[t]]$ whose $\ints/p^n$-subextension has first positive lower jump $h$ has vanishing KGB
obstruction precisely when $h \equiv -1 \pmod{m}$.  If this happens, 
the conjugation action of $\ints/m$ on $\ints/p^n$ is faithful, or
equivalently, $G$ is center-free (\cite[Proposition
5.9]{Ob:ll}). So groups $G$ of this form that are neither cyclic nor
center-free cannot be weak local Oort.

Furthermore, Green-Matignon showed that if $G$ contains an
abelian subgroup that is neither cyclic nor a $p$-group, 
then $G$ is not a weak local Oort group (see \cite[Proposition 3.3]{Gr:af}, which shows that no abelian group can be weak local Oort unless it is
cyclic or a $p$-group --- it is more or less trivial to see that if a
group contains a subgroup that is not weak local Oort, then the group
itself is not weak local Oort).  For a somewhat stronger statement, see \cite[Theorem 1.8]{CGH:ll}.

\subsection{The (differential) Hurwitz tree obstruction}\label{Shurwitzobstruction}

Recall that the (local) KGB obstruction comes from exploiting the fact that the
genus of a curve does not change when it is lifted to characteristic
zero.  However, if we have a local $G$-extension $k[[z]]/k[[t]]$ and a
lift of its corresponding HKG-cover, then only remembering the genus
of this cover and its subcovers means that we are actually throwing
out a great deal of other information.  In particular, we are
forgetting the $p$-adic \emph{geometry} of the branch locus (that is,
the distances of the branch points from each other).  It turns out
that these distances satisfy subtle constraints, and the
information about these constraints can be packaged in a combinatorial
structure called a \emph{Hurwitz tree}.  The local $G$-extension
$k[[z]]/k[[t]]$ will have a \emph{Hurwitz tree obstruction} if no Hurwitz tree exists that
properly reflects the higher ramification filtration of $k[[z]]/k[[t]]$.

Henrio gave the first major exposition of Hurwitz trees (\cite{He:ht}), dealing with the case of local $\ints/p$-extensions.  The concept was extended by Bouw and
Wewers to encompass local $\ints/p \rtimes \ints/m$-extensions (\cite{BW:ll}).  Later, it was
extended by Brewis and Wewers to arbitrary groups (\cite{BW:ac}).  A
detailed overview of the construction of Hurwitz trees for $\ints/p
\rtimes \ints/m$-actions is also given in \cite[\S7.3.1, \S7.3.2]{Ob:ll}.
We will give a much briefer overview below, and then we will mention
how one generalizes to the case of arbitrary local $G$-extensions.

\subsubsection{Building a Hurwitz tree}
Suppose $k[[z]]/k[[t]]$ is a local $G$-extension that lifts to a
$G$-extension $R[[Z]]/R[[T]]$, where $R$ is a characteristic zero complete discrete
valuation ring.  The key observation is
that we should think of $R[[Z]]$ as the ring of $R$-valued functions
on the \emph{$p$-adic open unit disc} (since, if $x$ is algebraic over
$R$ and has positive valuation, all power series converge on $x$).  Thus we will think of $\mc{D} := \Spec R[[Z]]$ as the open
unit disc.  By assumption, $G$ acts faithfully on $\mc{D}$ by $R$-automorphisms with no inertia above
a uniformizer of $R$ (that is, the action reduces to a faithful $G$-action on $\Spec k[[z]]$).  
The idea of a Hurwitz tree is to understand the geometry of this
action in combinatorial form.
We will write $\mc{D}_K$ for the generic fiber of $\mc{D}$, where $K =
\Frac(R)$.   Clearly, $G$ acts on $\mc{D}_K$.  We assume $R$ and $K$
are large enough for all ramification points of $\mc{D}_K \to \mc{D}_K/G$ to be defined over $K$. 

Let $\mc{D}' = \mc{D}/G = \Spec R[[T]]$, and write $\mc{D}'_K$
analogously to $\mc{D}_K$.  Let $y_1, \ldots, y_s$  (resp.\ $z_1,
\ldots, z_r$) be the branch (resp.\ ramification)
points of $\mc{D}_K \to \mc{D}'_K$.  The first step in building the Hurwitz
tree is to let $Y^{st}$ (resp.\ $Z^{st}$) be the \emph{stable model} of
$\proj^1_K$ corresponding to $(\mc{D}; y_1, \ldots, y_s)$ (resp.\
$(\mc{D}'; z_1, \ldots, z_r)$---we assume $r, s \geq 2$,
which is automatic as long as $G$ has nontrivial $p$-Sylow subgroup).  This is the minimal
semistable $R$-curve with generic fiber $\proj^1_K$ 
that separates the specializations of the $y_i$ (resp.\ $z_i$) and
$\infty$, where $\mc{D}_K$ (resp.\ $\mc{D}'_K$) $\subseteq \proj^1_K$ is viewed as the open unit disc centered 
at $0$ (see, e.g, \cite[Appendix A]{Ob:ll} for more details).  One way
to think about this is that we extend the coordinate $Z$ (resp.\ $T$)
on $\mc{D}_K$ (resp.\ $\mc{D}'_K$) to $\proj^1_K$.
Let $\ol{Y}$ (resp. $\ol{Z}$) be the special fiber of $Y^{st}$ (resp.\ $Z^{st}$).
Note that there is no natural map $Y^{st} \to Z^{st}$, but the group
$G$ does act naturally on the irreducible components of $\ol{Y}$ (via
lifting a point on an irreducible component to $\mc{D}_K$ and looking at
the specialization of its image under an element of $G$), and associates an irreducible component of $\ol{Z}$ to 
each irreducible component of $\ol{Y}$ (its ``image'' under the
``quotient map'').\footnote{All of this can be avoided by simply
  taking $Y^{st}$ and $Z^{st}$ to be semistable models of the discs
  themselves, in which case $Y^{st} \to Z^{st}$ would be the actual quotient map, but
  semistable models of discs are perhaps less familiar than semistable
  models of curves.}   Furthermore,
if $\ol{V}$ is an irreducible component of $\ol{Y}$ with generic point
$\eta_{\ol{V}}$ lying above an
irreducible component $\ol{W}$ of $\ol{Z}$ wth generic point $\eta_{\ol{W}}$, the map $\mc{D} \to \mc{D}'$ induces Galois extensions of complete local rings
$\hat{\mc{O}}_{Y^{st},  \eta_{\ol{V}}}/\hat{\mc{O}}_{Z^{st},\eta_{\ol{W}}}$.

The underlying tree of the Hurwitz tree is built from the \emph{dual
  graph} $\Gamma$ of $\ol{Z}$.
Specifically, the edges $E(\Gamma)$ and vertices $V(\Gamma)$ correspond to the 
irreducible components and nodes of $\ol{Z}$, respectively.  An edge connects two vertices if the corresponding node is the intersection of the two
corresponding components.   We append another vertex and edge $v_0$ and $e_0$ so that $e_0$ connects $v_0$ to the vertex corresponding to the 
component containing the specialization of $\infty$. Lastly, we append
a vertex $v_i$ for each branch point $y_i$, and connect it via an edge
$e_i$ to the vertex representing the irreducible component of $\ol{Z}$ to which $y_i$ specializes.  This gives us a tree $\Gamma'$.

To each edge $e \in E(\Gamma) \backslash \{e_0, \ldots, e_r\}$, we
attach a positive number $\epsilon_e$ equal to the \emph{thickness} of
the corresponding node.  That is, if the edge corresponds to a point
$\ol{z} \in \ol{Z}$, then the complete local ring
$\hat{\mc{O}}_{Z^{st}, \ol{z}}$ is isomorphic to $R[[U, V]]/(UV -
\rho)$ where $v(\rho) \in \rats$ (\cite{He:dc}), and we set $\epsilon_e = v(\rho)$.  
We set $\epsilon_{e_0}$ equal to $v(z_i)$, where
$z_i$ is a branch point of smallest possible valuation (thought of as
an element of $R$).  Such a branch point
will specialize to the same component of $\ol{Z}$ as $\infty$.  Lastly, we set
$\epsilon_{e_i} = 0$ for all $i > 0$.

Now, suppose the $p$-Sylow subgroup of $G$ is $\ints/p$.  In this case
we attach the following data to each vertex $v \in V(\Gamma')$.

\begin{itemize}
\item A nonnegative integer $\delta_v$ called the \emph{depth} at the
  vertex $v$, which is a measure of inseparability of the map $\mc{D}
  \to \mc{D}'$ above the irreducible component $\ol{W}$ of $\ol{Z}$
  corresponding to $v$ (for $v \in V(\Gamma)$).
Specifically, let $\eta_v$ be the generic point of this irreducible
component and let $\eta_w$ be the generic point of the (always unique) component of
$\ol{Y}$ above it.  The extension
$\hat{\mc{O}}_{Y^{st}, \eta_w} / \hat{\mc{O}}_{Z^{st}, \eta_v}$ has inertia group $\ints/p$ because
of its inseparable residue field extension, and
one can define its depth $\delta_v$ as in \cite[Appendix
B.1]{Ob:ll}.  One also sets $\delta_{v_0} = 0$ and $\delta_{v_i} = 1$
for $1 \leq i \leq r$. 
\item A meromorphic differential form $\omega_v$ on $\ol{W}$ 
associated to the extension $\hat{\mc{O}}_{Y^{st}, \eta_w} / \hat{\mc{O}}_{Z^{st}, \eta_v}$ called a \emph{differential
    datum} (or \emph{deformation datum}).  These are only associated
  to the vertices in $V(\Gamma)$.  By definition, they are either \emph{exact}
  (equal to $df$ for some $f \in k(\ol{W})$) 
  or \emph{logarithmic} (equal to $df/f$ for some $f \in k(\ol{W})$).  See \cite[Corollaire 1.8]{He:ht} or
  \cite[Appendix B.2]{Ob:ll}.  
\end{itemize}

The tree $\Gamma'$, along with the $\epsilon_e$, $\delta_v$, and
$\omega_v$ together form a \emph{Hurwitz tree}.  It turns out that
the data satisfy a host of conditions (see, e.g., \cite[Definition 3.2 and Proposition
3.4]{BW:ll} or \cite[Proposition 7.9]{Ob:ll}).  A (general) Hurwitz tree for a group $G$ with $p$-Sylow
subgroup $\ints/p$ is a tree $\mc{T}$, along with data $\epsilon_e$,
$\delta_v$, and $\omega_v$ satisfying these conditions (see
\cite[Definition 3.2]{BW:ll} or \cite[Definition 7.10]{Ob:ll}).

If $G$ is a general finite group of the form $P \rtimes \ints/m$ with
$P$ a $p$-group and $p \nmid m$, the Hurwitz tree construction of \cite{BW:ac} starts
in the same way, and one produces the same tree $\Gamma'$ and the same
edge thicknesses $\epsilon_e$.  However, one makes the following changes.

\begin{itemize}
\item Each vertex $v \in V(\Gamma')$ now has a \emph{decomposition
    conjugacy class $[G_v]$}, which is the conjugacy class (inside $G$) of the Galois group of $\hat{\mc{O}}_{Y^{st}, \eta_w} / \hat{\mc{O}}_{Z^{st}, \eta_v}$.
\item Instead of being a number, the depth $\delta_v$ for $v \in
  V(\Gamma')$ is now a \emph{$\rats$-valued character on $G$}.  It is
  induced from $G_v$ (see \cite[\S3.2]{BW:ac}).  It is equal to
  $0$ on $v_0$.
\item Each edge $e \in E(\Gamma')$ now has an \emph{Artin character} $a_e$,
  which is an (integral) character on $G$ (see \cite[\S3.3]{BW:ac}).
  In fact, the Artin character $a_{e_0}$ is the Artin character of the
  higher ramification filtration of the original local $G$-extension
  (this result is from \cite[Definition 3.5]{BW:ac}, and the
  definition of the Artin character can be found on \cite[p.\ 99]{Se:lf}). 
\item There are no differential data $\omega_v$.
\end{itemize}

\begin{remark}
The depth character is a useful way of encoding the depth information
for all of the $\ints/p$-subextensions of $R[[Z]]/R[[T]]$.
\end{remark}

\begin{remark}\label{Rartin1}
Similarly, the Artin character is a useful way of encoding the
differential data for all of the $\ints/p$-subextensions of
$R[[Z]]/R[[T]]$.  In the case where the $p$-Sylow subgroup $P$ of $G$
\emph{is} $\ints/p$ (so that we can define differential data and 
Artin characters), the
Artin character tells us about the zeroes and poles of the 
differential data.  Indeed, for an edge $e \in E(\Gamma)$
corresponding to a node $\ol{w}$ of $\ol{Z}$, the value of
$a_e$ on a nontrivial element of $P$ is exactly the order of the
divisor of $\omega_v$ at $\ol{w}$, where $v$ is one of the vertices incident to $e$ and thus corresponds to an
irreducible component $\ol{W}$ of $\ol{Y}$ passing through $\ol{w}$. 
\end{remark}

\begin{remark}\label{Rartin2}
One would like to have a full theory of differential data for general
Hurwitz trees.  In particular, for a vertex $v \in V(\Gamma')$, one
would like to associate a meromorphic differential form
$\omega_v(\chi)$ to each character $\chi$ of $G$.  The thesis of
Brewis (\cite[Chapter 1]{Brewisthesis}) shows how to use Kato's Swan
conductors to do this when
$\chi$ has rank $1$.  However, if $\chi$ has higher rank, one can
obtain a tensor product of differential forms rather than just one
form.  Furthermore, even when one does obtain just one form, it is not
clear exactly how to categorize what forms can be obtained (unlike in
the case where $G$ has $p$-Sylow subgroup of order $p$, in which case
the form is exact or logarithmic).  
\end{remark}

\begin{remark}
In \cite{He:ht}, \cite{BW:ll}, and \cite{Ob:ll}, the graph $\Gamma$
and thicknesses $\epsilon_e$ come from $Y^{st}$ instead of $Z^{st}$
(in the case where $G = \ints/p$, this has the effect of dividing the
$\epsilon_e$ by $p$).  We use $Z^{st}$ for our definitions since that is what the
more general Hurwitz trees of \cite{BW:ac} use.     
\end{remark}

\begin{example}\label{Ehurwitz}
Let $R$ be a complete discrete valuation ring containing $\zeta_p$,
and let $\lambda = \zeta_p - 1$, which has valuation $1/(p-1)$.  Pick
$N \in \nats \backslash p\nats$.
In Example \ref{Ezp}, it is shown that taking the integral closure $A$
of $R[[T]]$ in $\Frac(R[[T]])[W]/(W^p - 1 - \lambda^pT^{-N})$ gives a lift
of the local $\ints/p$-extension given by taking the integral closure
of $k[[t]]$ in $k((t))[y]/(y^p - y - t^{-N})$.  We write down the
(differential) Hurwitz tree for this lift, using the notation above freely.

The branch points $z_1, \ldots, z_{N+1}$ of $\mc{D}_K \to \mc{D}'_K$
are at $T = 0$ and at $T$ any $N$th root of $-\lambda^p$.  Since these points are all mutually equidistant, we have that
$\ol{Z}$ is isomorphic to $\proj^1_k$, and the graph $\Gamma$ is
simply a point, which we call $v$.   To build the Hurwitz tree
$\Gamma'$, we add a root vertex $v_0$ and vertices $v_1, \ldots,
v_{N+1}$ corresponding to the branch points, with each $v_i$ connected
to $v$ by an edge $e_i$.  The thickness $\epsilon(e_0)$ is just the
valuation $p/N(p-1)$ of $\sqrt[N]{-\lambda^p}$.  If $i > 0$, then
$\epsilon(e_i) = 0$.

By definition, we have $\delta_{v_0} = 0$ and $\delta_{v_i} = 1$ for
$1 \leq i \leq N+1$.  It turns out that $\delta_v = 1$ as well
(e.g., \cite[Definition 3.2(ii) and Lemma 3.3(ii)]{BW:ll}).  Lastly,
let $\alpha$ be an $N$th root of $-\lambda^p$.  Then, taking $X =
T/\alpha$, we have that $X$ reduces to a coordinate $x$ on $\ol{Z}$
(the specializations of the branch points to $\ol{Z}$ correspond to $x = 0$ and
all the $N$th roots of unity).
Since $1 + \lambda^pT^{-N} = 1 - X^{-N}$, it follows from
\cite[\S3.2]{BW:ll} (see also \cite[Corollaire 1.8(A)]{He:ht}) that
$$\omega_v = \frac{d(1 - x^{-N})}{1 - x^{-N}} = -N\frac{dx}{x(x^N - 1)}.$$ 
\end{example}

\subsubsection{The Hurwitz tree obstruction}\label{Shurwitzobstruction2}

In the context of a general Hurwitz tree, the way one gets an
obstruction is now clear.  Let $k[[z]]/k[[t]]$ be a local
$G$-extension whose higher ramification filtration has Artin character
$\chi$.  We say there is a \emph{Hurwitz tree obstruction} to lifting
if no Hurwitz tree for $G$ can be constructed having
Artin character $\chi$ on the vertex $v_0$.

\begin{example}\label{Ebrewiswewers}
Let $G = Q_{2^m}$, the generalized quaternion group of order
$2^m$, which can be presented as
$$G = \langle \sigma, \tau \mid \tau^{2^{m-1}} = 1, \tau^{2^{m-2}} = \sigma^2, \sigma\tau\sigma^{-1} = \tau^{-1}\rangle.$$ 
If $m = 3$, this is the standard quaternion group, and we assume $m
\geq 3$. 
Let $\ol{G} = G/\langle \tau^2 \rangle \cong \ints/2 \times \ints/2$.
Then $\ol{G}$ is generated by the images $\ol{\sigma}$ and $\ol{\tau}$
of $\sigma$ and $\tau$.  Consider the $\ol{G}$-action on $k[[z]]$
such that $$\ol{\sigma}(z) = \frac{1}{1+z} \text{ and } \ol{\tau}(z) =
\frac{1}{1+\mu z}$$ with $\mu \in k \backslash \FF_2$.  This action
can be lifted to a faithful $G$-action on $k[[z]]$, giving a local
$G$-extension (\cite[Lemma 2.10]{CGH:og}).  By \cite[Theorem
4.8]{BW:ac}, this local $G$-extension has a
Hurwitz tree obstruction to lifting.  In particular, $Q_m$ is \emph{not} a
local Oort group (cf.\ Theorem \ref{TKGB}).
\end{example}

\subsubsection{The differential Hurwitz tree obstruction}\label{Sdiffhurwitz}
What should be the definition of the differential Hurwitz tree
obstruction?  Well, if $k[[z]]/k[[t]]$ is a local $G$-extension whose
higher ramification filtration has Artin character $\chi$, there
should be a \emph{differential Hurwitz tree obstruction} to lifting 
if no \emph{differential} Hurwitz tree for $G$ can be constructed
having Artin character $\chi$ on the vertex $v_0$.  Of course, we have
not defined a differential Hurwitz tree (aside from when $G$ has
$p$-Sylow subgroup $\ints/p$, in which case a differential Hurwitz
tree is just our usual definition of a Hurwitz tree with differential data).
But it should be a Hurwitz tree enriched with meromorphic differential
forms attached to each vertex $v$ (aside from the leaves) such that
the divisors of these forms have some sort of compatibility with the
Artin characters, and such that the forms themselves are of a certain
type (which will in general be more complicated than just ``exact'' or
``logarithmic,'' see Remark \ref{Rartin2}).  

We elaborate somewhat in the case $G = (\ints/p)^n$.  Suppose we have
a Hurwitz tree for $G$ whose underlying graph has only \emph{one}
non-leaf vertex.  Suppose further that the Artin character for the
edge going from this vertex to the root of the tree (i.e., the vertex
with depth zero) gives the value $m+1$ for all nontrivial elements of
$G$.  In this case, the extra ``differential data'' should
be an $n$-dimensional $\FF_p$-vector space of logarithmic differential
forms, each with a (unique) zero at $\infty$ of order $m-1$ and $m+1$
simple poles.  Such a vector space is called an $E_{m+1, n}$ in the
introduction to \cite{Pa:ev}.  

\begin{example}\label{Epagotobstruction}
Pagot has shown (\cite[Th\'{e}or\`{e}me 2]{Pa:ev}) that for $p \geq
3$, no $E_{m+1, 2}$ exists when $m+1 = ap$ with $a \in \{1,2,3\}$, unless
$a = 2$ and $p = 3$ (Turchetti has also shown nonexistence for $a = 5$
and $p = 3$, see \cite[Theorem 0.3]{Tu:el}).  As a consequence of this, he shows that a local
$(\ints/p)^2$-extension $k[[z]]/k[[t]]$ where every degree
$p$-subextension of $k[[t]]$ has ramification jump $p-1$
(equivalently, where the only ramification jump for $k[[z]]/k[[t]]$ is $p-1$) cannot lift
to characteristic zero when $p \geq 3$ (\cite[Th\'{e}or\`{e}me 3]{Pa:ev}). The proof
proceeds by first observing that a Hurwitz tree for this lift
has only one non-leaf vertex (\cite[Theorem III, 3.1]{GM:op}---this corresponds to the branch points
of the lift all being \emph{equidistant} from each other), and then
invoking the nonexistence of an $E_{p,2}$. 

In fact, there is no KGB obstruction to lifting this type of extension
(\cite[Proposition 5.8]{Ob:ll}), and one can show that the Hurwitz tree obstruction
vanishes as well.  Thus it is justified to say that there is a
differential Hurwitz tree obstruction to lifting in this case.
\end{example}

\begin{remark}\label{Rhurwitzpositive}
In fact, differential Hurwitz trees can be used to get \emph{positive}
results for the local lifting problem.  See \S\ref{Shurwitz}.
\end{remark}

\section{Summary of present local lifting problem results}\label{Ssummary}

\subsection{Local Oort groups}\label{Slocaloort}

From Theorem \ref{TKGB} and the discussion following it, we know that
the only possible local Oort groups are cyclic, $D_{p^n}$, and $A_4$
for $p = 2$.   The \emph{local Oort conjecture} says that cyclic groups
are local Oort groups.  The local-global principle shows that the
local Oort conjecture implies the Oort conjecture, and the existence
of HKG-covers shows that the Oort conjecture implies the local Oort
conjecture. Here is the state of current knowledge.

\begin{itemize}
\item Cyclic groups $G$ are local Oort (i.e., the Oort conjecture is
  true).  This is the main result of Obus-Wewers and Pop
  (\cite{OW:ce}, \cite{Po:oc}).
Earlier cases had been proven by
Sekiguchi-Oort-Suwa (\cite{SOS:ask}, $v_p(|G|) \leq 1$) and
Green-Matignon (\cite{GM:lg}, $v_p(|G|) \leq 2$).
\item Bouw and Wewers showed that $D_p$ is local Oort
for $p$ odd (\cite{BW:ll}), and Pagot showed the same for $p=2$
(\cite{Pa:ev}, \cite{Pa:rc}).  
\item Obus proved that $D_9$ is local Oort for $p = 3$ (\cite{Ob:go}), 
and Weaver proved the same for $D_4$ and $p = 2$ (\cite{We:D4}).
  These are the only dihedral groups not of the form $D_p$ for which this is known.
\item The group $A_4$ was announced to be a local Oort group for $p =
  2$ by Bouw in \cite{BW:ll}.  A proof has been written up by Obus in \cite{Ob:A4}.  
\end{itemize}

The results above will be discussed further in \S\ref{Sapprox} and \S\ref{Smumford}. 

\subsection{Weak local Oort groups}\label{Sweaklocaloort}

We have seen in \S\ref{Sweak} that there are certain obstructions to
being a weak local Oort group.  On the positive side, aside from the
local Oort groups from \S\ref{Slocaloort}, we have the following
results.

\begin{itemize}
\item Matignon showed that $(\ints/p)^n$ is weak local Oort for all $p$ and
$n$ (\cite{Ma:pg}).  This is done using explicit methods (see \S\ref{Sexplicit}).   
\item Obus showed that $G = \ints/p^n \rtimes \ints/m$ is weak local Oort
  whenever it is center-free (\cite[Corollary 1.19]{Ob:go}).  This is
  also a necessary condition as long as $G$ is not cyclic, as we saw in \S\ref{Sweak}.
\item Brewis showed that the group $D_4$ is a weak
local Oort group for $2$ (\cite{Br:D4}).  Weaver's proof that $D_4$ is
local Oort for $2$ uses Brewis's result as a base case (see \S\ref{Sequichar}). 
\end{itemize}

Important open questions about the local lifting problem will be
discussed in \S\ref{Sopen}.

\section{Lifting techniques and examples}\label{Stechniques}
In this section, we will survey a variety of techniques that have
been used to construct lifts of local $G$-extensions. 

\subsection{Birational lifts and the different criterion}\label{Sbirational}
Usually, when dealing with Galois extensions of $k[[t]]$, it will be more convenient to deal with extensions of fraction fields than extensions of rings.
For instance, by Artin-Schreier theory, one knows that any $\ints/p$-extension $L/k((t))$ is given by an equation of the form $y^p - y = f(t)$.  
But writing down equations for the integral closure of $k[[t]]$ in $L$ is much more difficult.  So we will often want to think of a Galois ring extension in terms of the associated extension
of fraction fields.  In particular, we define a \emph{birational lift}
as follows:

\begin{definition}\label{Dbirational}
Let $k[[z]]/k[[t]]$ be a local $G$-extension.  A \emph{birational
  lift} over a complete characteristic zero discrete valuation ring $R$ with residue field
$k$ is a $G$-extension $M/\Frac(R[[T]])$ such that:
\begin{enumerate}
\item If $A$ is the integral closure of $R[[T]]$ in $M$, then the
  \emph{integral closure} of $A_k$ is isomorphic to (and identified
  with) $k[[z]]$ (equivalently, $\Frac(A_k) \cong k((z))$).
\item The $G$-action on $k((z)) = \Frac(A_k)$ induced from that on $A$ restricts to the given $G$-action on $k[[z]]$.
\end{enumerate}
\end{definition}

In fact, Garuti has shown (\cite{Ga:pr}) that any local $G$-extension has a birational lift to characteristic zero.

The following criterion, which saves one from the effort of making
explicit computations with integral closures, is extremely useful for seeing when a birational lift is actually a lift.

\begin{prop}[The different criterion, {\cite[I, 3.4]{GM:lg}}]\label{Pdifferent}
Suppose $A/R[[T]]$ is a birational lift of the local $G$-extension
$k[[z]]/k[[t]]$.  Let $K = \Frac(R)$, let $\delta_{\eta}$ be the
degree of the different $\mc{D}_{\eta}$ 
of $A_K/R[[T]]_K$ (i.e., the length of $A_K/\mc{D}_{\eta}$ as a
$K$-module), and let $\delta_s$ be the degree of the different
$\mc{D}_s$ of $k[[z]]/k[[t]]$ (i.e., the length of $A_k/\mc{D}_s$ as a
$k$-module).  Then $\delta_s \leq \delta_{\eta}$, and 
equality holds if and only if $A/R[[T]]$ is a lift of $k[[z]]/k[[t]]$
(that is, $A \cong R[[Z]]$). 
\end{prop}

\begin{remark}\label{Rbasechange}
Replacing $R$ and $K$ by finite extensions does not affect the degree
of $\mc{D}_{\eta}$ above, so we may assume that the ramified
ideals in $A_K/R[[T]]_K$ have residue field $K$.
\end{remark}

\begin{remark}\label{Requichar}
The different criterion is also valid when $R$ is an
\emph{equicharacteristic} complete discrete valuation ring (i.e., $R =
k[[\varpi]]$).  This will be used in \S\ref{Smumford}.
\end{remark}

\subsection{Explicit lifts}\label{Sexplicit}

Sometimes, the simplest way of lifting a local $G$-extension is to
write down explicit equations.  We give two examples in this
section.  

\begin{example}{\textbf{($\ints/p$-extensions)}}\label{Ezp}
The following argument shows that all local $\ints/p$-extensions lift
to characteristic zero.  It is a simplified version of arguments
originally from \cite{SOS:ask}, and can also be found in
\cite[Theorem 6.8]{Ob:ll}.  Since it is the most basic example, one
would be remiss not to include it here.

The key observation is that any
$\ints/p$-extension of a characteristic zero field containing a $p$th
root of unity is a \emph{Kummer extension}, given by extracting a
$p$th root.  The trick is then to assume $\zeta_p \in R$ and to find an element of $\Frac(R[[T]])$
such that normalizing $R[[T]]$ in the corresponding Kummer extension
yields the original local \emph{Artin-Schreier extension}.

Say that an element of a field $L$ of characteristic $p$ is a
\emph{$\wp$th power} if it is expressible as $x^p - x$ for $x \in L$.
By Artin-Schreier theory, any $\ints/p$-extension of $k((t))$ is given
by $k((t))[y]/(y^p - y - g(t))$, and is well defined up to adding a
$\wp$th power to $g(t)$.  In particular, we may assume that $g(t) \in
t^{-1}k[t^{-1}]$, as any element $u \in k[[t]]$ can be written as $x^p
- x$, where $x = -u - u^p - u^{p^2} - \cdots$.  Similarly, we may
assume that $g(t)$ has no terms of degree divisible by $p$.  
If $g(t) = t^{-N}h(t)$, where $h(t) \in k[t]$ has nonzero constant
term, then $h(t)$ is an $N$th power in $k((t))$, so replacing $t$ with
an $N$th root of $1/g(t)$ (which is a uniformizer), we may assume that $g(t) = t^{-N}$, with $p
\nmid N$.  

Given a local $\ints/p$-extension $k[[z]]/k[[t]]$, we may thus assume
without loss of generality that it is the integral closure of $k[[t]]$
in the Artin-Schreier extension of $k((t))[y]/(y^p - y- t^{-N})$
given by $t^{-N}$.  Let $R= W(k)[\zeta_p]$, let $\lambda = \zeta_p -
1$, and let $K = \Frac(R)$.
Then $v(\lambda^{p-1} + p) > 1$.  Consider the integral closure $A$ of $R[[T]]$ in the Kummer extension of $\Frac(R[[T]])$ 
given by 
\begin{equation}\label{Ezpform}
W^p = 1 + \lambda^p T^{-N}.
\end{equation}  Making the substitution $W = 1 + \lambda Y$, we obtain
$$(\lambda Y)^p + p\lambda Y  + o(p^{p/(p-1)})= \lambda^p T^{-N},$$ where $o(p^{p/(p-1)})$ represents terms with coefficients of valuation
greater than $p/(p-1)$.  This reduces to $y^p - y = t^{-N}$.  So we
have constructed a birational lift.  

The jump in the ramification filtration for $k[[z]]/k[[t]]$ occurs at
$N$ (Exercise!  This can be done explicitly by writing a uniformizer
in terms of $t$ and $y$).  By (\ref{Ebasicdifferent}), the degree of the different of $k[[z]]/k[[t]]$ is $(N+1)(p-1)$.
On the other hand, the generic fiber of $\Spec A \to \Spec R[[T]]$ is branched at exactly $N + 1$ points in the 
unit disc (at $T = 0$ and $T = \nu$ as $\nu$ ranges through the $N$th roots of $-\lambda^p$).
Since the ramification is tame, the degree of the different of $A_K/R[[T]]_K$ is 
$(N + 1)(p-1)$ as well.  By Proposition \ref{Pdifferent}, our birational lift is an actual lift.
\end{example}

\begin{remark}\label{Rgeneralcyclic}
It is easy to use Example \ref{Ezp} to show that all local
$\ints/pm$-extensions lift to characteristic zero, when $p \nmid m$.
See \cite[Proposition 6.3]{Ob:ll}.
\end{remark}

\begin{remark}\label{Rwhichuniformizer}
The case $N = 1$ is the local version of Example \ref{Ezplift} above.
Note that in this case, $y^{-1}$ is a uniformizer of $k[[z]]$, and we can
set $z = y^{-1}$.  Taking $Z = Y^{-1} = \lambda/(W-1)$, Remark \ref{Ranyuniformizer}
shows that $A$ can be written as $R[[Z]]$, once we verify that $Z \in A$.  This
is true because expanding out  the equation $(1 + \lambda Z^{-1})^p =
1 + \lambda^p T^{-1}$ coming from (\ref{Ezpform}) and muliplying both sides by $TZ^p/\lambda^p$ gives an
integral equation for $Z$ over $R[[T]]$.
\end{remark}

\begin{example}{\textbf{(Some $\ints/2 \times \ints/2$-extensions)}}\label{Ezpzp}
For odd $p$, it is an open problem in general to determine exactly
which local $\ints/p \times \ints/p$-extensions lift to characteristic
zero (see Example \ref{Enegativezpzp} and also \cite[Proposition
5.8]{Ob:ll}, which is an exposition of material in \cite[I, Theorem
5.1]{GM:lg}).  However, some local $\ints/p \times \ints/p$-extensions
can be lifted explicitly.  For example, suppose $p = 2$, and consider the local
$\ints/2 \times \ints/2$-extension $k[[z]]$ of $k[[t]]$ given by normalizing
$k[[t]]$ in $k((t))[y, w]/(y^2 - y - t^{-1}, w^2 - w - t^{-N})$ for
any odd $N > 1$.  Letting $R = W(\ol{\FF}_2)$, I claim that normalizing
$R[[T]]$ in $$L = \Frac(R[[T]])[U, V]/(U^2 - 1 - 4T^{-1}, V^2 - 1 -
4T^{-N})$$ gives an extension $A/R[[T]]$ lifting $k[[z]]/k[[t]]$.  Since the field extension
$L/\Frac(R[[T]])$ is the compositum of two $\ints/2$-extensions, each
giving rise to a lift of the component $\ints/2$-extensions of
$k[[t]]$ (see (\ref{Ezpform}), and note that $\lambda^p = 4$), we have
that $L/\Frac(R[[T]])$ certainly gives rise to a \emph{birational}
lift.  In order to show that it is actually a lift, we apply the
different criterion (Proposition \ref{Pdifferent}).  

Since the upper numbering is preserved under taking quotients
(\cite[IV, Proposition 14]{Se:lf}, the upper jumps of $k[[z]]/k[[t]]$ appear at $1$
and $N$.  This means that $G_0 = G_1 = \ints/2 \times \ints/2$, and
$G_2 = \cdots = G_{2N-1} = \ints/2$, with the rest of the filtration
trivial.  By (\ref{Ebasicdifferent}), the degree of the different
is equal to $2N + 4$.  

On the other hand, since ramification groups in characteristic zero
are cyclic, every ramification index of the $\ints/2 \times
\ints/2$-extension $A_K / R[[T]]_K$ is $2$, and thus each
ramified ideal of $R[[T]]_K$ with residue field $K$ (which, after a
finite extension of $K$, we assume is every ramified ideal) contributes $2$ to the degree of the
different.  These ramified ideals correspond to the zeroes and poles
of $1 + 4T^{-1}$ and $1 + 4T^{-N}$ in the open unit disc.  There are a
total of $N + 1$ zeroes and $1$ pole, showing that the degree of the
different is $2N + 4$.  We are done.
\end{example}

\begin{remark}
In fact, every local $\ints/2 \times \ints/2$-extension lifts to
characteristic zero, but writing down the lift is not generally as
straightforward as above.  For more on this, see \cite{Pa:ev} and \cite{Pa:rc}. 
\end{remark}

\begin{remark}
Matignon (\cite{Ma:pg}) has shown that $(\ints/p)^n$ is a weak local Oort group for
any $p$ and $n$, by writing down an explicit example and and explicit lift.
\end{remark}

For an example of a local $D_p$-extension for any odd $p$ with an explicit lift to
characteristic zero, see \cite[IV, Proposition 2.2.1]{GM:op} (or
\cite[Proposition 7.3]{Ob:ll} for the same example).  As in the case
of $\ints/2 \times \ints/2$, all local $D_p$-extensions lift to
characteristic zero (\cite{BW:ll}), but it is not in general easy to write down the lift
explicitly.

For an example of explicit lifts of some local $A_4$-extensions to characteristic zero, see \cite[Propositions 5.1, 5.2]{Ob:A4}.  The paper \cite{Br:D4}
gives examples of explicit lifts of some local $D_4$-extensions.

\subsection{Sekiguchi-Suwa Theory}\label{Ssekiguchi}
One potential way of obtaining explicit lifts for cyclic local
extensions is the \emph{Kummer-Artin-Schreier-Witt theory}, or
\emph{Sekiguchi-Suwa theory} (developed in \cite{SS:kasw1},
\cite{SS:kasw2}, with \cite{MRT:ss} being a nice survey).   Here, we will limit
ourselves to mentioning that \emph{Kummer-Artin-Schreier theory}, as
developed by Sekiguchi, Oort, and Suwa in \cite{SOS:ask}, gives an
explicit group scheme $\mc{G}$ defined over $\ints_p[\zeta_p]$ whose special fiber is $\Gp_a$ and whose generic fiber is
$\Gp_m$.  Furthermore, the theory exhibits the (more or less unique) degree $p$ isogeny on
$\mc{G}$ explicitly.  Any lift of a local $\ints/p$-extension (which is
Artin-Schreier) to a Kummer extension is a torsor under the kernel of
this isogeny, and knowing the explicit equations cutting out this
kernel leads one to discover the Kummer extension used in Example
\ref{Ezp}.  The Kummer-Artin-Schreier-Witt theory generalizes this story to
isogenies of degree $p^n$. We refer the reader to \cite[\S4.8]{Ob:ll}
for a brief exposition, and then to \cite{MRT:ss} if deeper knowledge
is desired.

We note that Green and Matignon were able to use the
Kummer-Artin-Schreier-Witt theory to show that $\ints/p^2$ is a local
Oort group (\cite{GM:lg}, or \cite[\S6.5]{Ob:ll} for an overview).
The equations involved in Kummer-Artin-Schreier-Witt theory for
$\ints/p^n$ become very complicated when $n
> 2$, and the theory has not been successfully applied to the local
lifting problem for these groups.

\subsection{Hurwitz trees}\label{Shurwitz}

The Hurwitz trees discussed in \S\ref{Shurwitzobstruction} have been
used to obtain positive results for the local lifting problem in the
case $G = \ints/p$ (\cite{He:ht}---of course, this is already proven
in Example \ref{Ezp}) and $G = \ints/p \rtimes \ints/m$
(\cite{BW:ll}. \cite{BWZ:dd}).  The process is outlined in some detail
in \cite[\S7.3.3 and \S7.3.4]{Ob:ll}, and we will not repeat it here.
We content ourselves with stating a (lightly paraphrased) version of
the theorem of Bouw, Wewers and Zapponi.

\begin{theorem}[\cite{BWZ:dd}, Theorem 2.1]\label{Tzpzm}
Suppose $p$ is a prime not dividing $m$. A nonabelian $\ints/p \rtimes
\ints/m$-extension $k[[z]]/k[[t]]$ lifts to characteristic zero iff
its KGB obstruction vanishes.  
\end{theorem}

\begin{remark}\label{Rdihedral}
By \cite[Proposition 5.9]{Ob:ll}, Theorem \ref{Tzpzm} is equivalent to stating that $k[[z]]/k[[t]]$
lifts to characteristic zero if and only if the ramification jump of
the $\ints/p$-subextension is congruent to $-1$ (mod $m$).  This
condition holds when $m = 2$ (e.g., \cite[Lemma
1.4.1(iv)]{Pr:fw}), so the dihedral group $D_p$ is local Oort for $p$ odd.
\end{remark}

\subsection{Successive approximation}\label{Sapprox}
The method of successive approximation gives a new approach to lifting local $G$-extensions, which has been
successful in the case that $G = \ints/p^n$ (\cite{OW:ce}) and $\ints/p^n \rtimes
\ints/m$, with $p \nmid m$ (\cite{Ob:go}).  The method has been
described in detail in \cite[\S6.6]{Ob:ll} when $G =
\ints/p^n$, and is similar when $G = \ints/p^n \rtimes \ints/m$
with $p \nmid m$, so as in \S\ref{Ssekiguchi}, we give only an
extremely brief overview.  

In both cases, one constructs a lift inductively.  Suppose $G = \ints/p \rtimes \ints/m$, and $k[[z]]/k[[t]]$ is
a local $G$-extension with vanishing KGB obstruction.  If $G$ is
cyclic (resp.\ non-abelian), then the extension lifts by Example
\ref{Ezp} and Remark \ref{Rgeneralcyclic} (resp.\ Theorem \ref{Tzpzm}).  Now, by
induction, assume that the $G/(\ints/p)$-subextension $k[[s]]/k[[t]]$ has
a lift to characteristic zero.  If $G$ is cyclic, we start by \emph{guessing} the form of a $\ints/p^n$-Kummer
extension in characteristic zero such that the
$\ints/p^{n-1}$-subextension is a lift of $k[[s]]/k[[t]]$, and the
full extension is a lift of \emph{some} local $\ints/p^n$-extension $k[[z']]/k[[t]]$
with smallest possible upper jump.  Our guess will generally be
incorrect, but one can show that it can be gradually deformed into
something closer to a lift (closeness of $A/R[[T]]$ to being a lift is
measured by the different of $A/R[[T]]$ localized at a uniformizer of $R$).  The
main result of \cite{OW:wr} shows that this deformation process
eventually stops, and we get a lift of $k[[z']]/k[[t]]$.  A further deformation gives a lift of
$k[[z]]/k[[t]]$, as desired.  If $G$ is metacyclic, the idea is the
same, but we try to lift $k[[z]]/k[[t]]$ from the beginning, without
going through the intermediate step involving the extension
$k[[z']]/k[[t]]$. 

The method of successive approximation gives rise to the following two
results.

\begin{theorem}[{\cite[Theorem 1.4]{OW:ce}}, {\cite[Theorem
    6.28]{Ob:ll}}]\label{Tcyclicnoessram}
Suppose $k[[z]]/k[[t]]$ is a local $\ints/p^n$-extension with upper
jumps $u_1, \ldots, u_n$, such that $u_{i+1} < pu_i + p$ for $1
\leq i \leq n-1$.  Then $k[[z]]/k[[t]]$ lifts to characteristic zero.
\end{theorem}

\begin{remark}\label{Rsuccapprox1}
\begin{enumerate}[(i)]
\item The criterion on the $u_i$ stated in \cite{OW:ce} and
  \cite{Ob:ll} is slightly different from what appears above, but it
  holds whenever the criterion above holds by \cite[Key Lemma
    4.15]{Po:oc}.  Pop calls the condition above \emph{no essential ramification}.
\item By Example \ref{Ejumpprogression}, one has $u_{i+1} \geq pu_i$
  for any local $\ints/p^n$-extension. 
\item In fact, in \cite{OW:ce}, one only needs that $u_{i+1} < pu_i + p$ for $2 \leq i
  \leq n-2$.  Since the full Oort conjecture has been proven in any
  case (using the techniques of \S\ref{Smumford}), this is
  not so important.
\end{enumerate}
\end{remark}

\begin{theorem}[{\cite[Theorem 1.14]{Ob:go}}]\label{Tmetacyclicnoessram}
Suppose $k[[z]]/k[[t]]$ is a local $\ints/p^n \rtimes
  \ints/m$-extension whose (unique) $\ints/p^n$-subextension has upper
  jumps $u_1, \ldots, u_n$ such that all $u_i \equiv -1 \pmod{m}$, that $u_{i+1} < pu_i + mp$ for $1 \leq i
  \leq n-1$, and that $u_1 < mp$.  If a certain criterion called the ``isolated
  differential data criterion'' is satisfied by $(u_1, \ldots, u_n)$,
  then $k[[z]]/k[[t]]$ lifts to charcteristic zero.
\end{theorem}

\begin{remark}\label{Rsuccapprox2}
\begin{enumerate}[(i)]
\item The isolated differential data criterion is defined in
  \cite[\S1.4 and Definition
  7.23]{Ob:go}.  The
  definition is somewhat technical, so we do not discuss it here
  beyond saying that it has to do with the existence of certain logarithmic
  differential forms on $\proj^1_k$, and can be phrased as asserting
  the existence of a solution to an equidetermined system of
  (non-linear!) equations over $k$.  In the context of Theorem
  \ref{Tmetacyclicnoessram}, 
  it is satisfied if $G = D_9$, if $G = D_{p^2}$ with $u_1 = 1$, or if
  $$(u_1, \ldots, u_n) = (m-1, p(m-1), \ldots, p^{n-1}(m-1))$$
  (\cite[Propositions 8.1, 8.2, 8.4]{Ob:go}).
\item The isolated differential data criterion is related to the
  differential Hurwitz trees from \S\ref{Sdiffhurwitz}.  It is
  somewhat stronger than what should be meant by ``there exists a
  differential Hurwitz tree for $G$ whose Artin character matches that
  of $k[[z]]/k[[t]]$.''
\item The condition that the jumps be congruent to $-1 \pmod{m}$
  is required because otherwise there is a local KGB obstruction to
  lifting (\cite[Proposition 5.9]{Ob:ll}).
\item The condition $u_{i+1} < pu_i  + mp$ is also called \emph{no
    essential ramification}.
\end{enumerate}
\end{remark}

\subsection{The ``Mumford method''}\label{Smumford}

A technique that has been tremendously successfully applied to the local
lifting problem in recent years is the method of deforming a local
$G$-extension \emph{within characteristic $p$} to one that has nicer
properties, showing that one can lift this nicer extension to
characteristic zero, and then showing that this implies the original
extension can be lifted.  Oort has called this the ``Mumford method''
because a similar idea, originally due to Mumford, was used to show that
all abelian varieties over algebraically closed fields of
characteristic $p$ lift to characteristic zero.  Namely, by the
Serre-Tate theory of canonical liftings, it was known that
\emph{ordinary} abelian varieties lift to characteristic zero (see,
e.g., \cite{Ka:ST}).  Norman and Oort (\cite{NO:ma}) showed that any
characteristic $p$ abelian variety deforms to an ordinary abelian
variety, that can then be lifted by the Serre-Tate theory.  They then
showed that this implies the original abelian variety can be lifted.
In this case, ``nice'' means ordinary.  In the case of the local
lifting problem, ``nice'' will be related to having limited
ramification in some sense.  This sense might vary depending on
the group; for $\ints/p^n \rtimes \ints/m$, ``nice'' will mean ``no
essential ramification.''

\subsubsection{Equicharacteristic deformations}\label{Sequichar}

Let $k[[z]]/k[[t]]$ be a local $G$-extension.  An
\emph{equicharacteristic deformation} of $k[[z]]/k[[t]]$ is a
$G$-extension $k[[\varpi, z]]/k[[\varpi, t]]$, where $\varpi$ is a transcendental
parameter, such that the $G$-action on $k[[\varpi, z]]$ reduces to the
original $G$-action on $k[[z]]$ modulo $\varpi$.  This is the
equicharacteristic version of lifting to characteristic zero (of
course, in the equicharacteristic case we have the \emph{trivial} deformation, so
the behavior is somewhat different)!   We think of $\varpi$ as the
deformation parameter, and the generic fiber of the deformation is $k[[\varpi,
z]][\varpi^{-1}]/k[[\varpi, t]][\varpi^{-1}]$.

\begin{remark}\label{Requichardisc}
It is useful to think of $k[[\varpi, t]][\varpi^{-1}]$ as the ring of functions on the \emph{open
unit disc} over $k((\varpi))$ with parameter $t$ (as an exercise to
understand this, think about which
elements $t - a$ are units and which are not, as $a$ ranges over
$k((\varpi))$).  Observe that $k((\varpi))[[t]]$ is strictly bigger
than $k[[\varpi, t]][\varpi^{-1}]$!
\end{remark}

Let us give a nontrivial example of an equicharacteristic
deformation, which is the inspiration for all examples of
equicharacteristic deformations we will mention (see Proposition \ref{Pequicharexamples}).

\begin{example}\label{Eequidef}
Let $k[[z]]/k[[t]]$ be the $\ints/p$-extension given by taking the
integral closure of $k[[t]]$ in $$k((t))[y]/(y^p - y - t^{-N}),$$ where
$N > p$ is not a multiple of $p$.  We
claim that the integral closure $\mc{A}$ of $k[[\varpi, t]]$ in $$k((\varpi,
t))[y]/(y^p - y - t^{-p}(t - \varpi)^{-N + p})$$ is an equicharacteristic
deformation.  Setting $\varpi = 0$ clearly yields $k((z))/k((t))$ after taking
fraction fields, but we must show that $\mc{A} \cong k[[\varpi, z]]$.

To do this we use the different criterion (Proposition
\ref{Pdifferent}) along with Remark \ref{Requichar}.  It suffices to
show that the degree $\delta_s$ of the different of the original $G$-extension is
equal to the degree $\delta_{\eta}$ of the different of the generic fiber of the
deformation.  We have $\delta_s = (N+1)(p-1)$ (see Example
\ref{Ezp}).  On the generic fiber, the two ramified ideals are $(t)$
and $(t - \varpi)$.  The function $g = t^{-p}(t - \varpi)^{-N + p}$ has a
pole of order $N-p$ when expanded out in $k((\varpi))((t - \varpi))$,
and thus the ideal $(t - \varpi)$ gives a contribution
of $(N - p + 1)(p-1)$ to $\delta_{\eta}$.  On the other hand, $g$ has a pole of order $p$ when expanded out in
$k((\varpi))((t))$.  Thus, by replacing $g$ with $g + x^p - x$ for
some $x \in k((\varpi, t))$ (which doesn't change the Artin-Schreier extension), we may assume that $g$ has a pole of
order less than $p$, and thus that $(t)$ contributes \emph{at most}
$p(p-1)$ to $\delta_{\eta}$.  So $\delta_{\eta} \leq (N+1)(p-1)$.  By
Proposition \ref{Pdifferent}, we in fact have equality, and thus
$\mc{A} \cong k[[\varpi, z]]$.
\end{example} 

\begin{remark}\label{Rlessram}
Notice that the ramification jumps on the generic fiber are smaller
than the ramification jumps of the originial extension.  In fact, based on
the example above, it is an easy exercise to show that for a local
$\ints/p$-extension, one can always find an equicharacteristic
deformation such that the ramification jumps on the generic fiber are
less than $p$.
\end{remark}

\subsubsection{Lifting via equicharacteristic deformations}\label{Eequicharlifting}

In order to apply the Mumford method, we need to show that being able
to lift the generic fiber of an equicharacteristic deformation to
characteristic zero allows us to do the same for the original local
$G$-extension.  First, we must say what we mean by ``being able to
lift the generic fiber.''  Take $k[[\varpi,
z]][\varpi^{-1}]/k[[\varpi, t]][\varpi^{-1}]$ and tensor over
$k((\varpi))$ with the algebraic closure $\ol{k((\varpi))}$.  We obtain a $G$-extension of
Dedekind $\ol{k((\varpi))}$-algebras, and localizing at any branched
maximal ideal gives a $G$-extension of $\ol{k((\varpi))}[[s]]$ for
some parameter $s$ (for instance, one could have $s = t$ or $s = t - \varpi$).  This is a local $G$-extension (with the field
$\ol{k((\varpi))}$ replacing $k$).  We say that the generic fiber
lifts to characteristic zero if all of the local $G$-extensions
obtained this way lift to characteristic zero.

The following theorem says more or less that being able to lift the
generic fiber of an equicharacteristic deformation implies being able
to lift the original local $G$-extension to charcteristic zero.  The
argument comes from \cite{Po:oc} and a conversation with Pop, but was
only written in \cite{Po:oc} for $G$ cyclic.
 The papers \cite{Ob:go} and \cite{Ob:A4} use similar arguments, but do not directly cite \cite{Po:oc} since they deal with
 non-cyclic groups.  
Our statement here is intended to be citeable for general $G$.  

\begin{theorem}\label{Tequicharsetup}
Suppose that $k[[z]]/k[[t]]$ is a local $G$-extension that admits an equicharacteristic
deformation whose generic fiber lifts to characteristic zero after
base change to the algebraic closure.  Then $k[[z]]/k[[t]]$ lifts to
characteristic zero.  
\end{theorem}

\begin{proof}
Let $k[[\varpi, z]]/k[[\varpi, t]]$ be an
equicharacteristic deformation of $k[[z]]/k[[t]]$.  Let $Y \to W = \proj^1_k$ be the HKG-cover
associated to $k[[z]]/k[[t]]$.  Let $\mc{W} = \proj^1_{k[[\varpi]]}$
with coordinate $t$.  There is a $G$-cover of flat relative $k[[\varpi]]$-curves
$\mc{Y} \to \mc{W} = \proj^1_{k[[\varpi]]}$ such that $\mc{Y}
\times_{\mc{W}} \Spec k[[\varpi,t]] \to \Spec k[[\varpi, t]]$
corresponds to $k[[\varpi, z]]/k[[\varpi, t]]$ via Spec, and that this
cover is unramified outside $\Spec k[[\varpi, t]]$ (this follows from
Pop's argument deducing \cite[Theorem 3.6]{Po:oc} from \cite[Theorem 3.2]{Po:oc}).
Write $\ol{\mc{Y}} \to \ol{\mc{W}}$ for the base
change of $\mc{Y} \to \mc{W}$ to the integral closure of $k[[\varpi]]$
in $\ol{k((\varpi))}$, and let $\ol{\mc{Y}_{\eta}} \to \ol{\mc{W}_{\eta}}$ be the generic fiber of $\ol{\mc{Y}} \to \ol{\mc{W}}$. 
Since the generic fiber of $k[[\varpi, z]]/k[[\varpi, t]]$ lifts to
characteristic zero after base change to the algebraic closure by assumption, the local-global principle tells us that $\ol{\mc{Y}_{\eta}} \to \ol{\mc{W}_{\eta}}$ 
lifts to a cover $\mc{Y}_{\mc{O}_1} \to \mc{W}_{\mc{O}_1}$ over some
characteristic zero complete discrete valuation ring $\mc{O}_1$ with residue field $\ol{k((\varpi))}$.  
Then, \cite[Lemma 4.3]{Po:oc} shows that we can ``glue'' the covers
$\ol{\mc{Y}} \to \ol{\mc{W}}$ and $\mc{Y}_{\mc{O}_1} \to \mc{W}_{\mc{O}_1}$
along the generic fiber of the former and the special fiber of the
latter, in order to get a cover $Y_{\mc{O}} \to W_{\mc{O}}$ defined over a \emph{rank two} characteristic zero valuation ring $\mc{O}$ with
residue field $k$ lifting $Y \to W$ (cf.\ \cite[p.\ 319, second
paragraph]{Po:oc}).  

We now show that $Y \to W$ lifts over a characteristic
zero \emph{discrete} valuation ring. Since the $G$-cover $Y_{\mc{O}} \to W_{\mc{O}}$ can be described using
finitely many equations, it descends to a cover $Y_A
\to W_A$ over some subring $A \subseteq \mc{O}$ that is \emph{finitely generated} over
$W(k)$.  Let $\mf{m} = A \cap \mf{m}_{\mc{O}}$, where
$\mf{m}_{\mc{O}}$ is the maximal ideal of $\mc{O}$. Then $A$ is a
domain, and $A/\mf{m} \cong k$.  Furthermore, the base change of $Y_A
\to W_A$ to $A/\mf{m}$ is the original cover $Y \to W$.  By Lemma
\ref{Ldescent}, there is an ideal $I \subseteq \mf{m} \subseteq A$
such that $A/I$ is a finite extension $R$ of $W(k)$.  Base changing
$Y_A \to W_A$ to $A/I$ gives a lift of $Y \to W$ over $R$. 
Applying the easy direction of the local-global principle, we obtain a
lift of $k[[z]]/k[[t]]$ over $R$, which concludes the proof. 
\end{proof}

\begin{remark}
In practice, the affine space constraint in Theorem
\ref{Tequicharsetup} does not cause trouble.  For example, suppose
$\mc{F}$ is the family of local $\ints/p$-extensions of $k[[t]]$ with ramification
jump at most $N$, for some $N$ not divisible by $p$.  These extensions
can be parameterized by polynomials $f$ in $t^{-1}$, where $f$ has
degree at most $N$ and no terms of degree divisible by $p$.  These
coefficients vary over an affine space, and the ramification jump is
$N$ on the complement of a hyperplane.

In general, if $G$ is a $p$-group, one can use \cite[Theorem 1.2 and
Proposition 2.1]{Ha:mp} to show that affine spaces parameterize local $G$-extensions of
$k[[t]]$ whose higher ramification filtrations are subject to a given bound.
\end{remark}

\begin{remark}
In the notation of the proof of Theorem \ref{Tequicharsetup}, if one
only cares about lifting the extension 
$k[[z]]/k[[t]]$ to characteristic zero without any uniformity over
$\mc{F}$ of the discrete valuation ring $R$, then one can invoke the
completeness of the theory of characteristic zero algebraically closed valued fields
with residue characteristic $p$ (\cite[III]{Ro:ct}) to see that the
$\ol{\Frac(\mc{O})}$ is an elementary extension of
$\ol{\Frac(W(k))}$, which means that any first-order sentence in the
language of characteristic zero algebraically closed valued fields
with residue characteristic $p$ that is true in $\ol{\Frac(\mc{O})}$
is also true in $\ol{\Frac(W(k))}$.  This can be shown to include the
sentence ``$k[[z]]/k[[t]]$ has a lift over the valuation ring.''  So
$k[[z]]/k[[t]]$ has a lift over the algebraic closure of $W(k)$, which
means it has a lift over some finite extension of $W(k)$.  This is a
lift over a discrete valuation ring. 
\end{remark}

\subsubsection{Consequences for specific groups}\label{Sconsequences}

For applications to the local lifting problem we need to know: How
nice can we hope to make local $G$-extensions via equicharacteristic
deformations?  Here is the current state of knowledge.  In all cases,
assume $k[[z]]/k[[t]]$ is a local $G$-extension.

\begin{prop}\label{Pequicharexamples}
\begin{enumerate}[(i)]
\item If $G = \ints/p^n$, then there is an equicharacteristic
  deformation of $k[[z]]/k[[t]]$ whose generic fiber has no essential ramification (i.e.,
  the upper jumps $u_1, \ldots, u_m$ at any ramified point of the generic fiber
  satisfy $u_{i+1} < pu_i + p$, see Remark \ref{Rsuccapprox1}(i)).
\item If $G = \ints/p^n \rtimes \ints/m$ is center-free, then there is
  an equicharacteristic deformation of $k[[z]]/k[[t]]$ whose
  generic fiber has one branch point with inertia
  group $G$ and no essential ramification (in the sense of Remark
  \ref{Rsuccapprox2}(iii)), and the rest of the branch points have
  cyclic inertia groups.  Furthermore, the upper jumps of the
  $\ints/p^n$-subextension of the generic
  fiber at the point with inertia group $G$ are congruent to the
  original upper jumps of this subextension modulo $mp$.
\item Suppose $G = A_4$ and all $\ints/2$-subextensions $k[[v]]/k[[u]]$
  of $k[[z]]/k[[t]]$, with $k[[v]] \neq k[[z]]$, have ramification jump
  $\nu \geq 6$.  Then there is an equicharacteristic deformation of
  $k[[z]]/k[[t]]$ whose generic fiber has one branch point with
  inertia group $G$ and corresponding ramification jump $\nu - 6$,
  and all of the other branch points have inertia group $\ints/2
  \times \ints/2$.
\item Suppose $G = D_4$, and let $\nu$ be the maximal number such that
  $|G^{\nu}| > 2$ for $k[[z]]/k[[t]]$ (we can think of $\nu$ as the
  ``second upper jump'' if we count jumps with multiplicity in case the
  ramification filtration jumps from $D_4$ straight to $\ints/2$).
  If $\nu > 1$, then there is an equicharacteristic deformation of
  $k[[z]]/k[[t]]$ such that the branch points of the generic fiber with
  inertia group $G$ have corresponding upper jump less than $\nu$, and
  all of the other branch points have inertia group $\ints/4$ or
  $\ints/2 \times \ints/2$.
\end{enumerate}
\end{prop}

\begin{proof}
Parts (i), (ii), (iii), and (iv) are due to Pop (\cite[Theorem
3.6]{Po:oc}), Obus (\cite[Proposition 3.1]{Ob:go}), Obus
(\cite[Proposition 3.2]{Ob:A4}), and Weaver (\cite{We:D4}), respectively.
\end{proof}

The equicharacteristic deformations above have significant
consequences for the local lifting problem.
\begin{theorem}\label{Tmainequichar}
\ \begin{enumerate}[(i)]
\item The Oort conjecture is true.
\item Theorem \ref{Tmetacyclicnoessram} holds even when the isolated
  differential data criterion is only satisfied for $(u_1', \ldots,
  u_n')$ where each $u_i' \equiv u_i \pmod{mp}$, where $u_1 < mp$, and
  where $pu_i \leq u_{i+1} < pu_i + mp$ for $i > 1$.
\item $A_4$ is a local Oort group.
\item $D_4$ is a local Oort group.
\end{enumerate}
\end{theorem}

\begin{proof}
To prove (i), we first reduce to the case $G = \ints/p^n$ by \cite[Proposition 6.3]{Ob:ll}.
By Proposition \ref{Pequicharexamples}(i) and Theorem \ref{Tequicharsetup}, it is enough to show
that cyclic extensions with no essential ramification lift to
characteristic zero.  But this is Theorem \ref{Tcyclicnoessram}.

Part (ii) follows from Proposition \ref{Pequicharexamples}(ii), 
Theorem \ref{Tequicharsetup}, and part (i) of this theorem.

To prove (iii), we proceed by induction on $\nu$, where $\nu$ is as in
Propsoition \ref{Pequicharexamples}(iii).
The base cases $\nu < 6$ can be handled explicitly, see
\cite[Propositions 5.1, 5.2]{Ob:A4}.  Then (iii) follows from
Proposition \ref{Pequicharexamples}(iii), Theorem
\ref{Tequicharsetup}, and the fact that $\ints/2
\times \ints/2$ is a local Oort group for $2$ (\cite{Pa:rc}).  

The proof of (iv) proceeds by induction on $\nu$ as in (iii), but here
the base case is $\nu = 1$ which is \cite[Theorem 4]{Br:D4}, we use
Proposition \ref{Pequicharexamples}(iv), and we
also use that $\ints/4$ and $\ints/2 \times \ints/2$ are local Oort
groups, by \cite{GM:lg} and \cite{Pa:rc} respectively. 
\end{proof}

\section{Approach using deformation theory}\label{Sdeformations}
The previous
sections of this chapter have been concerned with \emph{whether or
  not} a Galois branched cover or a local $G$-extension over $k$ lifts to
characteristic zero.  In this section, we will take a more general
approach and try to understand the \emph{deformation spaces} of Galois
branched covers and local $G$-extensions.  Our first goal is to
sketch a proof of a more refined version of the local-global
principle, due to Bertin and M\'{e}zard.  This reduces the global
deformation problem to a local one.  Then, we will give some results
on local deformation spaces.

We assume familiarity with deformation theory \`{a} la Schlessinger
(\cite{Sc:fa}) throughout, and we will not generally cite specific
results.  Another good reference for basic deformation theory is
\cite[Chapter 6]{FGAexplained}.

\subsection{Setup}\label{Sdefseup}

For our global deformation problem, we start with a smooth, projective, connected
$k$-curve $Y$ acted upon by a finite group $G$ of automorphisms (when
dealing with deformation theory, it will
often be more convenient to think of things this way than in terms of
branched covers).  Let $\hat{\mc{C}}$ be the category of complete
local noetherian $W(k)$-algebras with residue field $k$, and let
$\mc{C}$ be the full subcategory consisting of \emph{finite length} $W(k)$-algebras.  A
\emph{deformation of $(Y, G)$ over
$A$} is a relative smooth $A$-curve $Y_A$ with special fiber $Y$, such
that $G$ acts on $Y_A$ by $A$-automorphisms and this action restricts
to the original $G$-action on $Y$.   We define a global deformation
functor $$D_{gl}: \mc{C} \to \Sets$$ such that $D_{gl}(A)$ is the
set of $G$-equivariant isomorphism classes of deformations of $(Y, G)$ over
$A$.  If $f: A \to B$ is a $W(k)$-algebra homomorphism, then $D_{gl}(f): D_{gl}(A) \to D_{gl}(B)$ is
given by base change.

Similarly, suppose we have an injection of a finite group $G$ into
$\Aut_k k[[z]]$ (we will call this a \emph{local $G$-action}).  For $A \in \hat{\mc{C}}$, we define a
\emph{deformation of $G \hookrightarrow \Aut_k k[[z]]$ over $A$} to be
a map $G \to \Aut_A A[[Z]]$ such that the $G$-action on $A[[Z]]$
reduces to the given $G$-action on $k[[z]]$.  We define a local
deformation functor $$D_{loc}: \mc{C} \to \Sets$$ such that
$D_{loc}(A)$ is the set of classes of morphisms $G \to \Aut_A A[[Z]]$
lifting $G \hookrightarrow \Aut_k k[[z]]$, considered up to
conjugation by elements of $\Aut_A A[[Z]]$ reducing to the identity on $k[[z]]$.

Let $D: \mc{C} \to \Sets$ be a functor such that $D(k)$ has one element.  Recall
that a \emph{miniversal deformation ring} for $D$ is a ring $R \in \hat{\mc{C}}$ such that there
is a smooth (in particular, surjective) natural transformation $\xi:
\Hom_{W(k)}(R, \cdot) \to D$ of functors on $\mc{C}$ inducing an isomorphism
$\Hom_{W(k)}(R, k[\epsilon]/\epsilon^2) \stackrel{\sim}{\longrightarrow} D(k[\epsilon]/\epsilon^2)$.  
Note that most works we cite below call this a \emph{versal} deformation ring, but we will use the term miniversal so as
not to conflict with usages of ``versal'' in the literature that mean
only that $\Hom_{W(k)}(R, \cdot) \to D$ is smooth.
If $D$ is $D_{gl}$ or $D_{loc}$ above, we will simply say that $R$ is
a miniversal deformation ring for $(Y, G)$ or $G \hookrightarrow
\Aut_k k[[z]]$ respectively.  The ring $R$ is a \emph{universal
  deformation ring} if the natural transformation $\xi$ is an
isomorphism, and in this case the functor $D$ extends to
$\hat{\mc{C}}$ and is isomorphic to the extension of $\Hom_{W(k)}(R,
\cdot)$ to $\hat{\mc{C}}$.  We say that $D$ is
\emph{pro-representable}, and that the element of $D(R)$ corresponding
to the identity morphism in $\Hom_{W(k)}(R,R)$ is the universal deformation.
Miniversal and universal deformation rings are unique up to
isomorphism when they exist.  It is not hard to show, using
Schlessinger's criteria, that both $D_{gl}$ and $D_{loc}$ have
miniversal deformation rings.  Furthermore, if $g(Y) \geq 2$, then
$D_{gl}$ has a universal deformation ring, as deformations of $(Y,G)$
have finite automorphism groups, and thus no infinitesimal automorphisms.

The \emph{tangent space} to a functor $D: \mc{C} \to \Sets$ is defined to be
$D(k[\epsilon]/\epsilon^2)$.

\subsection{The local-global principle via deformation
  theory}\label{Sdeflocalglobal} 

It is clear that a global deformation gives rise to local
deformations.  More specifically, let $Y$ be a smooth, projective,
connected $k$-curve with an action of a finite group $G$, and let
$(Y_A, G)$ be a deformation of $Y$ over $A \in \mc{C}$.  Let
$\{x_1, \ldots, x_n\}$ be the set of branch points of $Y \to Y/G$, and
for each $i$, choose an inertia group $G_i \subseteq G$ of a
ramification point $y_i$ above $x_i$ (it does not matter which point
is chosen). As we have seen at the beginning of
\S\ref{Slocalglobal}, the deformation $(Y_A, G)$ gives rise to
deformations of $G_i \hookrightarrow \Aut_k \hat{\mc{O}}_{Y, y_i}$
over $A$.  In particular, if $D_{gl}$ is the deformation functor
associated to $(Y, G)$ and $D_{loc_i}$ is the deformation functor
associated to $G_i \hookrightarrow \Aut_k \hat{\mc{O}}_{Y, y_i}$, we
have natural transformations
$$D_{gl} \to D_{loc_i}$$ for all $i$.  If $R_{gl}$ and $R_i$ are the
respective miniversal deformation rings, we obtain morphisms $R_i \to
R_{gl}$.  

If $D_{loc}$ is the direct product of the deformation functors
$D_{loc, i}$, then by general categorical principles, the miniversal
deformation ring of $D_{loc}$ is $R_{loc} := \widehat{\bigotimes}_i R_i$.  The
natural transformation of functors $D_{gl} \to D_{loc}$ thus gives
rise to a ring homomorphism
$$\phi: R_{loc} \to R_{gl}.$$
The statement of the deformation theoretic local-global principle is
as follows.

\begin{theorem}[{\cite[Corollaire 3.3.5]{BM:df}}]\label{Tdeflocalglobal}
We have $R_{gl} \cong R_{loc}[[U_1, \ldots, U_M]]$ for
some $M$, and the map $R_{loc} \to R_{gl}$ is the natural
inclusion.  In other words, the map of deformation functors $D_{gl}
\to D_{loc}$ is smooth.
\end{theorem}

The following corollary generalizes the local-global principle of
Theorem \ref{Tlocalglobal} to complete noetherian local rings when the
genus of the curve is at least $2$.

\begin{corollary}\label{Clocalglobal}
Let $Y$ be a smooth projective $k$-curve with an action of a finite
group $G$.  Let $A$ be a complete local noetherian ring with residue
field $k$.  Let $y_1, \ldots, y_r$ be the ramification points of $Y
\to Y/G$ with inertia groups $G_1, \ldots, G_r$.  Then $(Y, G)$ has a deformation over $A$ if and only if 
each local $G_i$-action on $\hat{\mc{O}}_{Y, y_i}$ has a deformation
over $A$.  
\end{corollary}

\begin{proof}
The ``only if'' direction is obvious.  So suppose that each local
$G_i$-action has a deformation over $A$.  Without loss of generality,
suppose that the set $y_1, \ldots, y_s$, for some $s \leq r$, consists
of one ramification point above each branch point of $Y \to Y/G$.  By miniversality, we have a homomorphism
from $\widehat{\bigotimes}_{i=1}^s R_i  = R_{loc} \to A$, where $R_i$ is the
miniversal deformation ring corresponding to the local $G_i$-action on
$\hat{\mc{O}}_{Y, y_i}$.  By Theorem \ref{Tdeflocalglobal}, this gives
rise to a morphism $R_{gl} \to A$, by choosing the images of $U_1,
\ldots, U_M$ in the maximal ideal of $A$ arbitrarily (these choices
parameterize the different global deformations with given local
behavior).  This gives rise to a global deformation over $A$.    
\end{proof}

\subsubsection{Brief sketch of the proof of Theorem \ref{Tdeflocalglobal}}\label{Sdefproofsketch}

Maintain the notation of \S\ref{Slocalglobal}.  The key to the proof of Theorem \ref{Tdeflocalglobal} is to understand
the \emph{tangent spaces} and \emph{obstruction spaces} to the
relevant functors $D_{gl}$ and $D_{loc, i}$.  

By \cite[Th\'{e}or\`{e}me 2.2]{BM:df}, the tangent space to
$D_{loc, i}$ is isomorphic to $H^1(G_i, \Theta_i)$, where $\Theta_i$ is the
$G_i$-module $k[[z_i]]\frac{d}{dz_i}$ of formal derivations at the
ramification point $y_i$ (here $z_i$ is a local parameter of $Y$ at
$y_i$).  The obstruction space associated to $D_{loc, i}$ is $H^2(G_i,
\Theta_i)$ (\cite[Remarque 2.3]{BM:df}).  

Describing the global counterparts requires equivariant cohomology.  If $X$ is a
finite-type $k$-scheme with $G$-action and $\mc{F}$ is a
($G$, $\mc{O}_X$)-module, then we can define the \emph{equivariant
  cohomology groups} $H^q(G, \mc{F})$ to be
the right derived functors of the left-exact functor $\Gamma(X,
\mc{F})^G$.  If $\mc{T}_Y$ is the tangent sheaf of $Y$ with its
natural $G$-action, then the tangent and obstruction spaces associated
to $D_{gl}$ are $H^1(G, \mc{T}_Y)$ and $H^2(G, \mc{T}_Y)$,
respectively (\cite[Propositions 3.2.1, 3.2.3]{BM:df} --- see
\cite[Definition 6.1.21]{FGAexplained} for the definition of an
obstruction space).  

In fact, one can show that the map $\phi: D_{gl} \to D_{loc}$ of deformation
functors induces a \emph{surjection} $d\phi: H^1(G, \mc{T}_Y) \to
\bigoplus_i H^1(G_{x_i}, \Theta_i)$ on tangent spaces, whose kernel is the
tangent space for the functor of \emph{locally trivial deformations}
(\cite[Lemme 3.3.1]{BM:df}).  Furthermore, $\phi$ induces an
\emph{isomorphism} of obstruction spaces $H^2(G, \mc{T}_Y) \to
\bigoplus_i H^2(G_i, \Theta_i)$ (\cite[Lemme 3.3.2]{BM:df}).  A
standard deformation theory argument now shows that $\phi$ is
\emph{smooth}.    

By definition, $\Hom_{W(k)}(R_{gl}, \cdot)$ and $\Hom_{W(k)}(R_{loc}, \cdot)$
have the same tangent spaces as $D_{gl}$ and
$D_{loc}$, respectively.  The obstruction theories are the same as well, by
\cite[Lemma 6.3.3]{FGAexplained}.  So the same deformation theory argument shows that
$\Hom_{W(k)}(R_{gl}, \cdot)$ is smooth over $\Hom_{W(k)}(R_{loc}, \cdot)$.  Theorem \ref{Tdeflocalglobal} now follows from
\cite[Proposition 2.5(i)]{Sc:fa}.

\begin{remark}
Bertin and M\'{e}zard have proven a similar local-global principle for
equivariant deformations of reduced curves that are locally complete
intersections, but not necessarily smooth (\cite[Th\'{e}or\`{e}me
4.3]{BM:df2})
\end{remark}

\subsection{Examples of local miniversal deformation rings}\label{Sversal}
In view of Theorem \ref{Tdeflocalglobal}, understanding global
deformations is tantamount to calculating miniversal deformation rings
for local $G$-actions.  Since we will be dealing with one local action
at a time in this section, we will use $R_{loc}$ to mean the
miniversal deformation ring of a local $G$-action (this should not
cause confusion with the usage of $R_{loc}$ in
\S\ref{Sdeflocalglobal}).  If $G = \ints/m$ with $p \nmid m$, then it
is an easy exercise to see that $R_{loc} = W(k)$.  The following theorem of Bertin and M\'{e}zard gives information on
the miniversal deformation ring $R_{loc}$ for local $\ints/p$-actions.

\begin{theorem}[{\cite[Th\'{e}or\`{e}mes 4.2.8, 4.3.7,
 5.3.3]{BM:df}}]\label{Tversal}
Let $k[[z]]/k[[t]]$ be a local $\ints/p$-extension with upper jump
$N$, and let $R_{loc}$ be the miniversal deformation ring of the
corresponding local $\ints/p$-action.
\begin{enumerate}[(i)]
\item If $p = 2$, then $R_{loc} = W(k)[[x_1, \ldots, x_{(N+1)/2}]]$.
\item If $N = 1$ and $p = 3$, then $R_{loc} = W(k)$.
\item If $N = 1$ and $p > 3$, then $R_{loc}$ is isomorphic to $W(k)[[X]]/\psi(X)$, where
$$\psi(X) = \sum_{i=0}^{(p-1)/2} \binom{p-1-i}{i} (-1)^i(X+4)^{(p-1)/2 -
  i},$$ and is not formally smooth.
\item If $N > 1$ and $p$ is odd, then there is a surjection
$$R_{loc} \to W(k)[\zeta_p][[x_1, \ldots, x_{\lfloor (N+1)/p
  \rfloor}]].$$  Furthermore,
the Krull dimension of $R_{loc}$ equals $1 + \lfloor (N+1)/p \rfloor$. 
\end{enumerate}
\end{theorem}
 
\begin{remark}\label{ROSS}
Theorem \ref{Tversal}(iv) shows that the formal spectrum of the
miniversal deformation ring contains a smooth component of maximal dimension.
Bertin and M\'{e}zard call this the ``Oort-Sekiguchi-Suwa'' component.
\end{remark}

\begin{remark}
We see that when $p = 2$, the deformation space is formally smooth.
This is somewhat surprising, considering that the obstruction space $H^2(\ints/2, \Theta)$ is
nontrivial (\cite[Proposition 4.1.1]{BM:df}). In any case, combining
this result with Theorem \ref{Tdeflocalglobal} shows that the
deformation space for a hyperelliptic curve in characteristic $2$ is
formally smooth (\cite[Remarque 3.3.8]{BM:df}).  This is originally a
result of Laudal and L{\o}nsted (\cite{LL:dc})
\end{remark}

\begin{example}\label{E31}
Suppose char$(k) = 3$.  Consider the local $\ints/3$-extension
$k[[z]]/k[[t]]$ given by the integral closure of $k[[t]]$ in $$k((t))[y]/(y^3 - y -
t^{-1}),$$ with Galois action generated by $\sigma(y) = y+1$ (which
leads to $\sigma(z) = z/(z+1)$, if we set the uniformizer $z$ equal
to $y^{-1}$).  As we have seen from Example \ref{Ezp}, this can be lifted
over $R := W(k)[\zeta_3]$ by taking the integral closure $A$ of $R[[T]]$
in $$\Frac(R[[T]])[W]/(W^3 - (1 + \lambda^3T^{-1})),$$ where $\lambda =
\zeta_3 - 1$.  Here, $Z := \lambda/(W-1) \in A$ is a lift of $z$ and $A = R[[Z]]$ (Remark
\ref{Rwhichuniformizer}).  Furthermore
$\sigma(W) = \zeta_3W$ and therefore $\sigma(Z) = Z/(Z+\zeta_3)$ (where we abuse
notation and use $\sigma$ for the lift of $\sigma$ to $R[[Z]]$).  

By Theorem \ref{Tversal}(ii), however, this lift is
definable over $W(k)$!  To see how to do this, let 
$$V = \frac{\bar{\lambda} - \lambda W}{W - 1} \in R[[Z]],$$ where
$\bar{\lambda} \neq \lambda$ is Galois conjugate to $\lambda$.  One calculates $V =
\zeta_3Z - \lambda$.  If $v$ is the reduction of $V$, then $v = z$ and
we have $R[[Z]] = R[[V]]$ by Remark \ref{Ranyuniformizer}. 
A straightforward calculation yields $$\sigma(V) = \frac{V-3}{V-2},$$
which is defined over $W(k)$ and reduces to the original action
$\sigma(z) = z/(z+1)$.  So $R[[V]]/R[[T]]$ with this action is a lift
of the original local $\ints/3$-extension that is conjugate to the original lift $R[[Z]]/R[[T]]$.
\end{example}

There is not a great deal known about miniversal deformation rings of
local $G$-extensions when $G$ is not cyclic.  One case where there has
been some progress is the \emph{weakly ramified case} (i.e., when 
the second lower ramification group $G_2$ is trivial).  One reason
such extensions are of interest is the following result of Nakajima (\cite[Theorem
2(i)]{Na:pr}): Any local extension that arises from a Galois cover $X \to \proj^1_k$ where
$X$ is an \emph{ordinary} curve (i.e., the $p$-rank of
Jac$(X)$ is the genus of $X$) is weakly ramified.  This in turn
implies that $G \cong (\ints/p)^t \rtimes \ints/m$ for some $t$ and $p
\nmid m$, by basic ramification theory (\S\ref{Sramification}).  

\begin{theorem}[{\cite[Theorem 4.5]{CK:ed}}]\label{Tweakdimensions}
Let $k[[z]]/k[[t]]$ be a weakly ramified local $G$-extension with
miniversal deformation ring $R_{loc}$ and equicharacteristic miniversal
deformation ring $S_{loc} := R_{loc}/p$.  Suppose $G = (\ints/p)^t
\rtimes \ints/m$, with $p \nmid m$.  If $p \neq 2, 3$, then $S_{loc}$
has dimension $t/s - 1$, where $s = [\FF_p(\zeta_n) : \FF_p]$.  
\end{theorem}

\begin{remark}
In fact, Cornelissen and Kato compute $S_{loc}$ explicitly
(\cite[\S4.4]{CK:ed}).  The case where $G \cong \ints/p$ in Theorem \ref{Tweakdimensions}
is covered by Theorem \ref{Tversal}(i)--(iii).
\end{remark}

\begin{remark}
The corresponding HKG-cover to a weakly ramified local extension
$k[[z]]/k[[t]]$ is genus zero (exercise!).  This means that, after an
appropriate change of variables, the
relevant automorphisms of $k[[z]]$ can be expressed as fractional
linear transformations.  This is heavily exploited in calculating
deformation rings, and is one of the reasons the assumption of weak
ramification makes these calcuations more tractable. 
\end{remark}

\begin{remark}
Suppose $f: X \to \proj^1_k$ is a weakly ramified Galois cover, where
$g(X) \geq 2$ and char$(k) > 3$.  Let $n$ be the number of branch
points of $f$, and $w$ the number of those that are wildly branched.
Using the local-global principle (Theorem \ref{Tdeflocalglobal}), Cornelissen and Kato
show that the dimension of the equicharacteristic global miniversal
deformation ring is $$3g(X) - 3 + n + \sum_{i=1}^w
\frac{t_i}{s(n_i)},$$ where the inertia group above point $i$ is
$(\ints/p)^{t_i} \rtimes \ints/n_i$, and $s(n_i) = [\FF_p(\zeta_{n_i})
: \FF_p]$ (\cite[Main Algebraic Theorem]{CK:ed}).  The $t_i/s(n_i)$
terms can be thought of as ``correction'' terms for the wild
ramification (note that these terms vanish if the cover is tamely
ramified and one recovers the classical characteristic zero result).
In \cite{CK:ed}, this is applied to Drinfeld modular curves, and
comparison is made to an \emph{analytic} deformation functor.
Exploring this would take us somewhat far afield, so we direct the interested reader
to the (very readable) paper \cite{CK:ed}.
\end{remark}

Other papers dealing with equicharacteristic deformation of wildly
ramified covers include \cite{PZ:pr}, \cite{Pr:fw}, \cite{FM:cs}
(non-Galois case!), and \cite{CK:es}.

In addition to calculating the equicharacteristic miniversal deformation
rings for weakly ramified local extensions, we can also calculate the
general miniversal deformation ring (and thus,
whether the extension lifts to characteristic zero).

\begin{theorem}[{\cite[Th\'{e}or\`{e}me 1.2]{CM:rr}}]\label{Tweakcharacteristic}
If $R_{loc}$ is the local miniversal deformation ring of a weakly (but
wildly) ramified local $G$-extension, then the characteristic of
$R_{loc}$ is either $0$ or $p$.  In
particular, the characteristic of $R_{loc}$ is $0$ when $G \in \{\ints/p, D_p, A_4\}$,
and is $p$ otherwise.
\end{theorem}

\begin{remark}\label{Rfullweakcalculations}
The exact formulas for the miniversal deformation rings above are given in
\cite[Corollaire 4.1]{CM:rr}.
\end{remark}

\begin{example}\label{Enop2}
Consider the smooth projective curve $Y$ with affine equation $$(y^p -
y)(x^p - x) = 1$$ (\cite[p.\ 240]{CM:rr}).  The group $G = (\ints/p)^2
\rtimes D_{p-1}$ acts on this as follows: the generators $\sigma$ and
$\tau$ of the normal $(\ints/p)^2$ subgroup send $x$ to $x+1$ and $y$
to $y+1$, respectively, an element $\alpha$ of order $2$ in a copy $H$
of $D_{p-1}$ inside $G$ exchanges $x$ and $y$, and a generator
$\beta$ of
the order $p-1$ cyclic normal subgroup of $H$ sends $y$ to $cy$ and
$x$ to $c^{-1}x$, for some $c$ generating $\FF_p^{\times}$. 

Viewing $Y$ as a $\ints/p$-cover of the projective $x$-line, we see that it is
ramified above the $p$ poles of $1/(x^p - x)$, and the ramification is
weak in each case (it is unramified at $\infty$, as $1/(x^p - x) = 0$
at $x = \infty$).  In fact, the inertia group $I$ inside $G$ at the
ramification point above $x= 0$ is generated by $\tau$ and $\beta$,
and is thus isomorphic to $\ints/p \rtimes \ints/(p-1)$.  Since the
lower numbering respects subgroups by definition, the local action of $I$ at this
point is weakly ramified.  This, combined with the local-global
principle, gives a KGB obstruction to
lifting the $G$-cover $Y \to Y/G$ to characteristic zero when $p > 3$ (\cite[Proposition
5.9]{Ob:ll}).  By Theorem \ref{Tweakcharacteristic}, $Y \to Y/G$ cannot even lift to
characteristic $p^n$ for any $n > 1$. 
\end{example}

\begin{remark}
In the example above, if $p \geq 41$, then one can also use the
Hurwitz bound to show that lifting is impossible.
\end{remark}

Next, we give a result on whether miniversal deformation rings
are in fact \emph{universal}.  Unsurprisingly, the
strongest results are in the weakly ramified case.

\begin{prop}[{\cite[Theorems on p.\ 879]{BC:ww}, \cite{BCK:ad}}]\label{Puniversal}
Let $k[[z]]/k[[t]]$ be a local $G$-extension with
miniversal deformation ring $R_{loc}$ and equicharacteristic miniversal
deformation ring $S_{loc} = R_{loc} / p$.  
\begin{enumerate}[(i)]
\item If $k[[z]]/k[[t]]$ is weakly ramified, then $R_{loc}$ is not universal if and only if char$(k) = 2$ and $G$
is $\ints/2$ or $(\ints/2)^2$.    
\item If $k[[z]]/k[[t]]$, is weakly ramified, then $S_{loc}$ is not
  universal for equicharacteristic deformations if and only if
  char$(k) = 2$ and $G \cong \ints/2$.
\item If $G \cong \ints/5$ and $k[[z]]/k[[t]]$ has ramification jump
  $2$ (so necessarily char$(k) = 5$), then $R_{loc}$ is universal. 
\end{enumerate}
\end{prop}

Lastly, we mention a potential application of understanding miniversal and
universal deformation rings, due to Byszewski.  Suppose
$k[[z]]/k[[t]]$ is a local $G$-extension, and $H$ is a \emph{normal}
subgroup of $G$ with corresponding subextension $k[[z]]/k[[s]]$.  One
would like to be able to ``stack'' a lift of the $H$-extension
$k[[z]]/k[[s]]$ above a lift of the $G/H$-extension
$k[[s]]/k[[t]]$ to form a lift of $k[[z]]/k[[t]]$.  Of course, this is
in general difficult (otherwise the Oort conjecture would follow from
the $\ints/p$ case).  However, we have the following result.

\begin{theorem}[{\cite[Definition 2.1, Theorem 2.8]{By:df}}]\label{Tdeftower}
In the context above, let $R_G$ and $R_H$ be the
respective miniversal deformation rings for the corresponding local
extensions, with $D_G$ and $D_H$ the respective
deformation functors. There is a natural $G/H$-action on $D_H$, and a natural
transformation of functors from $D_G$ to $D_H^{G/H}$ coming from restriction.  If $p
\nmid |G/H|$ and $R_H$ is universal, then this natural transformation
of functors is an isomorphism, and $R_G$ is universal.  In particular,
$R_G \cong R_H/I_H$, where $I_H$ is the ideal generated by all $gx -
x$ for $g \in G$ and $x \in R_H$ (universality of $R_H$ gives a
natural $G/H$-action on $R_H$).
\end{theorem}

In the weakly ramified case, Theorem \ref{Tdeftower} gives an alternate
way of recovering some of the miniversal deformation rings from Remark
\ref{Rfullweakcalculations} (\cite[Proposition 2.10]{By:df}).

\section{Open Problems}\label{Sopen}
Some open problems are collected below.  Anything called a ``conjecture''
is something that I \emph{strongly} believe to be true.  If I am less
confident, I will use the word ``question.''  Some of the questions
are open-ended.  Unless otherwise mentioned, $k$ is an
algebraically closed field of characteristic $p$.
  
\subsection{Existence of local lifts}

Recall the the list of local Oort groups is (at most) the cyclic
groups, the dihedral groups $D_{p^n}$, and $A_4$ for $p=2$.
The most basic question about the local lifting problem is 

\begin{question}\label{Sstrongoort}
Is each group above a local Oort group?
\end{question}

As we have seen in \S\ref{Ssummary}, this question is only open for
$D_{p^n}$ where $n > 1$ (and is known for $D_4$ and $D_9$).  It has
been referred to as the ``stong Oort conjecture'' (\cite[Conjecture
1.2]{CGH:og}).

For groups that have obstructions to being local Oort (for instance, the KGB
obstruction), we can still ask if the KGB obstruction is the only
thing that goes wrong.  The following conjecture of mine
generalizes Theorem \ref{Tmetacyclicnoessram}.
\begin{conjecture}\label{Cobus}
The KGB obstruction is the only obstruction to the local lifting
problem for local $G$-extensions where $G$ has \emph{cyclic} $p$-Sylow
subgroup.  In particular, $D_{p^n}$ is a local Oort group when $p$ is odd.
\end{conjecture}
Recall that the local KGB obstruction vanishes exactly when all the
upper jumps $u_1, \ldots, u_n$ of the $\ints/p^n$-subextension of the
local $G$-extension $k[[z]]/k[[t]]$ are congruent to $-1 \pmod{m}$.
By Theorem \ref{Tmainequichar}(ii), one way of proving Conjecture
\ref{Cobus} is by showing that the isolated differential data
criterion holds for all such $(u_1, \ldots, u_n)$ where $u_1 \leq m_p$ and
$pu_i \leq u_{i+1} \leq pu_i + mp$.  Since the criterion is about
solving equidetermined systems of equations over an algebraically
closed field (Remark \ref{Rsuccapprox2}), its truth is
plausible. Furthermore, Conjecture \ref{Cobus} is true whenever the $p$-Sylow subgroup of
$G$ is $\ints/p$, by Theorem \ref{Tzpzm}.

The following question is much more speculative.
\begin{question}\label{Cpgroup}
Is every $p$-group a weak local Oort group?
\end{question}
Any $p$-group $G$ is a so-called \emph{Green-Matignon group}
(\cite[Definition 1.7]{CGH:ll}), which means, by \cite[Theorem
1.8]{CGH:ll}, that there is at least one local
$G$-extension whose Bertin obstruction vanishes (\cite{Be:ol} --- this
is a weaker obstruction than the KGB obstruction).  
However, the set of $p$-groups is too vast and our evidence in \S\ref{Ssummary} too
sparse to venture an opinion on the truth of Question
\ref{Cpgroup}.  Indeed, it is not even known for $Q_8$!  On the other
hand, with the proof of Brewis and Wewers that
$D_4$ is weak local Oort, one can be more confident in the following
conjecture.

\begin{conjecture}\label{CD2n}
The group $D_{2^n}$ is a weak local Oort group for $2$.
\end{conjecture}

Lastly we mention the following question (see \S\ref{Sdiffhurwitz}).

\begin{question}\label{Qdiffhurwitz}
For what groups $G$ can one give an explicit definition of the differential Hurwitz tree obstruction?
\end{question}
This should not be extremely difficult when the $p$-Sylow subgroup of
$G$ is abelian, and Example \ref{Epagotobstruction} provides some
guidance.  It may be trickier to obtain a clean expression for the
differential Hurwitz tree obstruction for arbitrary groups $G$.

\subsection{Rings of definition}

The proofs that $\ints/p$ and $\ints/p^2$ are local Oort groups
give explicit formulas for the lifts, and in the case of $\ints/p$
(resp.\ $\ints/p^2$), it is shown that lifting is possible over
$W(k)[\zeta_p]$ (resp.\ $W(k)[\zeta_{p^2}]$) by \cite{SOS:ask} (resp.\
\cite{GM:lg}), where $k$ is the field
of the local extension.  These explicit formulas
come from the Sekiguchi-Suwa theory (\S\ref{Ssekiguchi}).  

\begin{question}\label{Qoortspecific}
Can all local $\ints/p^n$-extensions over $k$ be lifted over $W(k)[\zeta_{p^n}]$?
\end{question}

Sekiguchi-Suwa theory fundamentally takes place over
$W(k)[\zeta_{p^n}]$, so if one could somehow lift all
$\ints/p^n$-extensions using Sekiguchi-Suwa theory, then one would get
a positive answer to Question \ref{Qoortspecific}.  Alternatively, one
could try to be more careful about the coefficients that show up in
the current proof (\S\ref{Sapprox}), but this seems quite hard.  We
mention that a ``weak'' version of Question
\ref{Qoortspecific}, has a positive answer.  Namely, for any $n$,
there exists a local $\ints/p^n$-extension over $k$ that can
be lifted over $W(k)[\zeta_{p^n}]$ (\cite[Theorem 4.2]{Gr:dc}, and
implicit in \cite[p.\ 280]{GM:op}).

\subsection{Moduli/deformations/geometry of local lifts}

We begin with a question of Gunther Cornelissen (\cite{AutQuestions}).

\begin{question}\label{Qwhichchar}
Is there a local $G$-extension over $k$ that lifts to a ring of
characteristic $p^n$ for $n > 1$, but does not lift to characteristic
zero?  That is, can the miniversal deformation ring of a local
$G$-extension have characteristic other than $0$ or $p$?
\end{question}

As we have seen in Theorem \ref{Tweakcharacteristic}, this is not possible for a
\emph{weakly ramified} local $G$-extension.

We can also ask about obstructions to lifting to non-prime characteristic.
\begin{question}\label{Qlargecharobstruction}
Is it possible to write down any general obstruction to lifting a
local $G$-extension to characteristic $p^n$ (for some $n > 1$), in the
spirit of the KGB or Hurwitz tree obstructions?
\end{question}

Of course, if the answer to Question \ref{Qwhichchar} is negative,
then obstructions to lifting to characteristic zero work equally well
as obstructions to lifting to characteristic $p^n$ for $n > 1$.

In another direction, we have seen in Example \ref{E31} that the
relationship between Artin-Schreier equations, Kummer equations, and deformation
rings of local $\ints/p$-actions is not straightforward
(although Bertin and M\'{e}zard's proof of Theorem \ref{Tversal}(iv)
is based on a partial understanding).  Having a Kummer equation for a lift of
a $\ints/p$-extension allows one to write down its Hurwitz tree and
the geometry of its branch locus. 

\begin{question}\label{Qdefvsartinschreier}
In the case where $G = \ints/p$, can one give an explicit link between
deformation parameters and Hurwitz trees of lifts?  To what extent can
Hurwitz trees be used to distinguish isomorphism classes of lifts?
Can Hurwitz trees be used to identify whether or not a deformation lies in the
Oort-Sekiguchi-Suwa component (Remark \ref{ROSS})?   
\end{question}

An answer to the following question would reduce the inexplicitness
of the Mumford method.
\begin{question}\label{Qgeometry}
Suppose a local $G$-extension is shown to lift to characteristic zero using
the Mumford method via a specific equicharacteristic
deformation.  Can one say anything about the geometry of the branch
locus of the resulting lift?  In particular, can one say anything
about the resulting Hurwitz tree?
\end{question}
It would be natural to attack Question \ref{Qgeometry} by first
assuming $G \cong \ints/p$, where the structure of Hurwitz trees is well-understood.\\

Saidi has called the following conjecture the ``Oort conjecture revisited'' (\cite{Sa:fl}).
\begin{conjecture}\label{Coorttower}
Local cyclic extensions are liftable in towers.  That is, given a local
$G$-extension with $k[[z]]/k[[t]]$ with $G$ cyclic and a lift $R[[S]]/R[[T]]$ of a
subextension $k[[s]]/k[[t]]$ to characteristic zero, there is a lift
of $k[[z]]/k[[t]]$ to characteristic zero containing $R[[S]]/R[[T]]$
as a subextension.
\end{conjecture}
The proof of the Oort conjecture in \cite{OW:ce} already proceeds by induction, and
one could in theory prove Conjecture \ref{Coorttower} by making the
induction process more flexible.  In particular, the induction
argument of \cite{OW:ce} only
works if one can lift a subextension so that the branch points all
have high enough valuation, see \cite[Theorem 3.4(i)]{OW:ce}.  If one
could remove the valuation restriction, Conjecture \ref{Coorttower}
would follow (and one would additionally get a proof of the Oort conjecture
without using the Mumford method).  Removing this restriction directly
seems more promising than trying to use deformation theory in towers
as in \cite{By:df} (see discussion before Theorem
\ref{Tdeftower}).  In any case, though, it would be interesting to
understand the deformation theory of local $\ints/p^n$-actions for $n
> 1$. 

\subsection{Non-algebraically closed residue fields}

Throughout this entire paper, we have considered the local lifting
problem over an \emph{algebraically closed} field of characteristic
$p$.  There is no reason that one cannot consider local $G$-extensions
$k[[z]]/k[[t]]$ where $k$ is an arbitrary field of characteristic $p$,
and try to lift them to a characteristic zero mixed characteristic
local domain \emph{with residue field $k$}.  The following is a
question of Oort.

\begin{question}\label{Qnonalgclosedlifting}
Does there exist a characteristic $p$ field $k$ and a cyclic branched cover of
curves defined over $k$ that does not lift over a characteristic
zero local normal domain $R$ with residue field $k$?  What if $k$ is perfect, or finite?  What if
we relax the assumption that $R$ is normal?
\end{question}
One can also ask the above question in the local context, i.e.,
whether there is a local cyclic extension over a characteristic $p$
field that does not lift over any
characteristic zero local normal domain.  It is not clear to me
whether this question is necessarily equivalent to Question
\ref{Qnonalgclosedlifting} (i.e., whether there is a local-global
principle in this context).
 
\appendix
\section{Some algebraic preliminaries}\label{Salgebraic}

\subsection{Homological Algebra}\label{Shomological}
\begin{lemma}[cf. {\cite[\S3, Lemma 1.1]{He:thesis}}]\label{Lhomological}
Let $R$ be a complete local noetherian ring with residue field $k$
and maximal ideal $\mf{m}$.
Suppose $0 \to M_1 \stackrel{u}{\to} M_2 \stackrel{v}{\to} M_3 \to 0$ is a complex of $\mf{m}$-adically separated
$R$-modules, with $M_1$ and $M_2$ complete, and $M_2$ and $M_3$ flat
over $R$.  For $i = 1,2,3$, let $\ol{M}_i = M_i / \mf{m} M_i$.  If $0 \to
\ol{M}_1 \stackrel{\ol{u}}{\to} \ol{M}_2 \stackrel{\ol{v}}{\to} \ol{M}_3 \to 0$ is exact, then so is $0 \to
M_1 \to M_2 \to M_3 \to 0$.
\end{lemma}

\begin{proof}
To prove exactness on the right, take $\beta_0 \in M_3$.  There exists
$\alpha_0 \in M_2$ and $\beta_1 \in M_3$ such that $v(\alpha_0) =
\beta_0 - m_1 \beta_1$, with $m_1 \in \mf{m}$.  Similarly, define $\alpha_n \in M_2$ and
$\beta_{n+1} \in M_3$ so that $v(\alpha_n) = \beta_n - m_{n+1}
\beta_{n+1}$ for some $m_{n+1} \in \mf{m}$.  Letting $\alpha = \sum_{n=0}^\infty m_1 \cdots m_n \alpha_n$,
the separatedness of $M_3$ yields that $v(\alpha) = \beta$.  

To prove exactness on the left, let $N = \ker(u)$.  It is a closed
submodule of $M_1$, thus separated.  
Since $0 \to N \to M_1 \to u(M_1) \to 0$ is exact and $u(M_1)$, being
torsion-free, is flat, we have that $0 \to \ol{N} \to \ol{M}_1 \to
\ol{u(M_1)} \to 0$ is exact.  So $\ol{N} = 0$, and $N = \mf{m} N$.  Since
$N$ is separated, we have $N = \bigcap_{n=1}^{\infty} \mf{m}^n N = 0$.

Exactness in the middle follows from surjectivity of $\ol{M}_1 \to
\ol{\ker(v)}$ and the result on right-exactness above.
\end{proof}

\begin{remark}\label{Risom}
Taking $M_1 = 0$ in Lemma \ref{Lhomological}, one sees immediately that
if $v: M_2 \to M_3$ is a morphism of $\mf{m}$-adically separated flat
$R$-modules, with $M_2$ complete, then $v$ is surjective (resp.\ an
isomorphism) if and only if its reduction modulo $\pi$ is.
\end{remark}

\begin{prop}\label{Pflat}
Let $X$ be a regular, connected scheme of dimension $\leq 2$ with
function field $K(X)$, let $L/K(X)$ be a finite extension, and let
$f: Y \to X$ be the normalization of $X$ in $L$.  Then $f$ is flat.
\end{prop}

\begin{proof}
Since $f$ is an integral morphism, $\dim Y = \dim X \leq 2$.  Since $Y$ is normal, it is
$S_2$ by Serre's criterion and thus Cohen-Macaulay.  Since flatness is
local, it suffices to check that the integral extension $\mc{O}_{Y, y}/\mc{O}_{X, f(y)}$ is
flat for all $y \in Y$.  By \cite[Theorem 18.16]{Ei:ca}, it suffices
to show that $\mc{O}_{Y, y}/\mf{m}_{f(y)}$ is zero-dimensional.  But this
is true because $\kappa(f(y))$ is a field and $\mc{O}_{Y,y}/\mf{m}_{f(y)}$ is integral over $\kappa(f(y))$.
\end{proof}

\subsection{Complete local rings}\label{Scomplete}
\begin{lemma}\label{Lcohen}
If $x$ is a closed point on a smooth curve $X/k$, then
$\hat{\mc{O}}_{X, x} \cong k[[t]]$.
\end{lemma}

\begin{proof}
Since $\hat{\mc{O}}_{X, x}$ is a complete discrete valuation ring in
equal characteristic with residue field $k$, this follows, e.g., from
\cite[II, Theorem 2]{Se:lf}.
\end{proof}

\begin{lemma}\label{Llocalring}
If $X_R$ is a curve defined over a complete local noetherian ring
$R$ with residue field $k$ and $x$ is a smooth closed point of $X_R$, then the complete local
ring $\hat{\mc{O}}_{X_R, x}$ is isomorphic to a power series ring
$R[[T]]$.
\end{lemma}

\begin{proof}
Let $X$ be the special fiber of $x$.  Then $\hat{\mc{O}}_{X, x} \cong
k[[t]]$ (Lemma \ref{Lcohen}).  Consider a homomorphism $\phi: R[[T]] \to \hat{\mc{O}}_{X_R, x}$ where
$T$ is sent to any lift of $t$.  The assumptions of Remark \ref{Risom}
are satisfied and the map $\phi$ is an isomorphism
modulo the maximal ideal of $R$  We conclude that $\phi$ is an isomorphism.
\end{proof}

\begin{remark}\label{Ranyuniformizer}
The proof above shows that if $T'$ is any element of $R[[T]]$ such
that $T'$ reduces to a uniformizer of $k[[t]]$, then $R[[T']] = R[[T]]$.
\end{remark}

\subsection{Ramification theory}\label{Sramification}
The following facts are from \cite[IV]{Se:lf}.
Recall that if $k$ is a field of characteristic $p$, one can form the
ring $W(k)$ of \emph{Witt vectors} over $k$.  If $k$ is perfect, this
is the unique complete characteristic zero discrete valuation ring with residue field $k$
and uniformizer $p$.  

Let $F$ be a complete DVF with algebraically closed residue field of
characteristic $p \geq 0$ and
uniformizer $\pi$.  If $L/F$ is a finite $G$-Galois extension, then 
$G$ is of the form $P \rtimes \ints/m$, where $P$ is a $p$-group
and $m$ is prime to $p$ (if $p = 0$, then $P$ is trivial).  In particular, $G$ is solvable.  The group $G$ has a filtration $G 
= G_0 \supseteq G_i$ ($i \in \reals_{\geq 0}$) defined by
$$g \in G_i \leftrightarrow v\left(\frac{g(\pi)}{\pi} - 1\right) \geq i+1$$ (here
$v$ is defined so that $v(\pi) = 1$.
There is also a filtration $G \supseteq G^i$ for the upper numbering ($i \in
\reals_{\geq 0}$) given by $G^i = G_{\psi(i)}$, where $\psi$ is the
inverse of the \emph{Herbrand function} $\varphi$, given by $\varphi(u)
= \int_0^u dt/[G_0 : G_t]$.  If $i \leq j$, then $G_i \supseteq G_j$ and $G^i \supseteq G^j$.  
The subgroup $G_i$ (resp.\ $G^i$) is known as the \emph{$i$th higher ramification group for the lower numbering (resp.\ the upper 
numbering)}.  

One knows that $G_0 = G^0 = G$, and that $G_1 = G^{\frac{1}{m}} = P$
(in particular, if $p = 0$ then $G_1$ is trivial).   
For sufficiently large $i$, $G_i = G^i = \{id\}$.   Any $i$ such that $G^i \supsetneq G^{i + \epsilon}$ for all $\epsilon > 0$ is
called an \emph{upper jump} of the extension $L/F$.  Likewise, if $G_i \supsetneq G_{i+\epsilon}$ for $\epsilon > 0$, then $i$ is called a
\emph{lower jump} of $L/F$.  If $i$ is a lower (resp.\ upper) jump, $i > 0$, and $\epsilon > 0$ is sufficiently small,
then $G_i/G_{i + \epsilon}$ (resp.\ $G^i/G^{i + \epsilon}$) is an
elementary abelian $p$-group.  The lower jumps are clearly all
integers.  The Hasse-Arf theorem says that the upper jumps are
integers whenever $G$ is \emph{abelian} (in general, the upper jumps
need only be rational).   The extension $L/F$ is called \emph{tamely ramified} if $G_1 =
\{id\}$ (equivalently, $G \cong \ints/m$), and \emph{wildly ramified}
otherwise.  

\begin{example}\label{Etametransition}
Suppose $L/F$ is a $G$-Galois extension as above, with residue field of characteristic $p >
0$.  Let $M$ be the subextension corresponding to $P \leq G$.  By the
definition of the lower numbering, 
we have $P_i = G_i$ for $i > 0$.  By the definition of
the upper numbering, we have $P^i = G^{i/m}$.  In particular, if $P$
is abelian then the upper jumps for $L/F$ lie in $\frac{1}{m}\ints$.  
\end{example}

The degree $\delta$ of the different of a $G$-extension $L/F$ is given by the formula
\begin{equation}\label{Ebasicdifferent}
\delta = \sum_{i=0}^{\infty} (|G_i| - 1).
\end{equation}  
Note that $\delta \geq |G|-1$, with strict
inequality holding if and only if $L/F$ is wildly ramified.

\begin{example}\label{Ejumpprogression}
Suppose $k[[z]]/k[[t]]$ is a $\ints/p^n$-extension with \emph{upper
  jumps} $u_1 < \cdots < u_n$.  
For $i > 1$, we have $u_i \geq pu_{i-1}$, with $p \nmid
u_i$ whenever strict inequality holds (see, e.g., \cite[Theorem 1.1]{Ga:ls}).
\end{example}

Now, let $f: Y \to X$ be a degree $d$ branched cover of curves over an
algebraically closed field $k$.  The cover is called tamely (resp.\ wildly)
ramified at a point $y$ if the corresponding extension
$\hat{\mc{O}}_{Y,y}/\hat{\mc{O}}_{X, f(y)}$ is.  
The \emph{Riemann-Hurwitz formula} (\cite[IV, \S2]{Ha:ag}) states that 
$$2g_Y - 2 = d(2g_X - 2) + \sum_{y \in Y} \delta_y,$$ where $\delta_y$
is the degree of the different of
$\hat{\mc{O}}_{Y,y}/\hat{\mc{O}}_{X, f(y)}$, and $g_Y$ (resp.\ $g_X$)
is the genus of $Y$ (resp.\ $X$).  The Riemann-Hurwitz formula,
combined with (\ref{Ebasicdifferent}), 
shows that if $Y \to X$ is a wildly ramified cover of curves over $k$, then the genus of $Y$ is higher than it would be 
if the cover had the same ramification points and indices, but was in
characteristic zero.

\subsection{Miscellaneous}\label{Smisc}

The following lemma is used in the proof of Theorem
\ref{Tequicharsetup}.

\begin{lemma}\label{Ldescent}
Let $k$ be an algebraically closed field of characteristic $p$, and
let $A$ be a finitely generated $W(k)$-algebra that is a domain.  If
$\mf{m} \subseteq A$ is an ideal such that $A/\mf{m} \cong k$ as a $W(k)$-algebra, then
there is an ideal $I \subseteq \mf{m} \subseteq A$ such that $A/I$ is
a finite extension of $W(k)$.
\end{lemma}

\begin{proof}
Embed $\Spec A$ into $\aff^n_{W(k)} \subseteq \proj^n_{W(k)}$ for some
$n$, and let $Z$ be the projective
closure of $\Spec A$ in $\proj^n_{W(k)}$.  Let $x \in \Spec A
\subseteq Z$ be the point
corresponding to $\mf{m}$.  Since $A$ is a domain, $x$ is in the
closure of the generic fiber of $\Spec A$, and thus of $Z$.  By
\cite[II, \S8, Theorem 1]{Mu:rb}, $x$ is the specialization of some
geometric point on the generic fiber of $Z$.  Let $y \in Z$ be the image of this
point.  Since $x$ does
not lie in the hyperplane at infinity, neither does $y$.  So $y$ is in
the generic fiber of $\Spec A$.  Since the closure $\ol{\{y\}}$ of $y$ is finite
over $\Spec W(k)$ and contains $x$, taking $I = I(\ol{\{y\}})
\subseteq A$ gives the
desired ideal.  
\end{proof}

\section*{Acknowledgements}
I thank Frans Oort for guidance and help in preparing this chapter,
and Florian Pop for useful conversations.

\addcontentsline{toc}{section}{Bibliography}
\bibliographystyle{alpha}
\bibliography{main}

\newcommand{\etalchar}[1]{$^{#1}$}
\begin{thebibliography}{FGI{\etalchar{+}}05}

\bibitem[Abh57]{Ab:ca}
Shreeram Abhyankar.
\newblock Coverings of algebraic curves.
\newblock {\em Amer. J. Math.}, 79:825--856, 1957.

\bibitem[BC09]{BC:ww}
Jakub Byszewski and Gunther Cornelissen.
\newblock Which weakly ramified group actions admit a universal formal
  deformation?
\newblock {\em Ann. Inst. Fourier (Grenoble)}, 59(3):877--902, 2009.

\bibitem[BCK12]{BCK:ad}
Jakub Byszewski, Gunther Cornelissen, and Fumiharu Kato.
\newblock Un anneau de d\'eformation universel en conducteur sup\'erieur.
\newblock {\em Proc. Japan Acad. Ser. A Math. Sci.}, 88(2):25--27, 2012.

\bibitem[Bel79]{Be:ge}
G.~V. Bely{\u\i}.
\newblock Galois extensions of a maximal cyclotomic field.
\newblock {\em Izv. Akad. Nauk SSSR Ser. Mat.}, 43(2):267--276, 479, 1979.

\bibitem[Ber98]{Be:ol}
Jos{\'e} Bertin.
\newblock Obstructions locales au rel\`evement de rev\^etements galoisiens de
  courbes lisses.
\newblock {\em C. R. Acad. Sci. Paris S\'er. I Math.}, 326(1):55--58, 1998.

\bibitem[BM00]{BM:df}
Jos{\'e} Bertin and Ariane M{\'e}zard.
\newblock D\'eformations formelles des rev\^etements sauvagement ramifi\'es de
  courbes alg\'ebriques.
\newblock {\em Invent. Math.}, 141(1):195--238, 2000.

\bibitem[BM06]{BM:df2}
Jos{\'e} Bertin and Ariane M{\'e}zard.
\newblock D\'eformations formelles de rev\^etements: un principe local-global.
\newblock {\em Israel J. Math.}, 155:281--307, 2006.

\bibitem[Bre08]{Br:D4}
Louis~Hugo Brewis.
\newblock Liftable {$D\sb 4$}-covers.
\newblock {\em Manuscripta Math.}, 126(3):293--313, 2008.

\bibitem[Bre09]{Brewisthesis}
Louis Brewis.
\newblock Ramification theory of the $p$-adic open disc and the lifting
  problem.
\newblock Ph.D. Thesis, available at
  http://webdoc.sub.gwdg.de/ebook/dissts/Ulm/Brewis2009.pdf, 2009.

\bibitem[BW06]{BW:ll}
Irene~I. Bouw and Stefan Wewers.
\newblock The local lifting problem for dihedral groups.
\newblock {\em Duke Math. J.}, 134(3):421--452, 2006.

\bibitem[BW09]{BW:ac}
Louis~Hugo Brewis and Stefan Wewers.
\newblock Artin characters, {H}urwitz trees and the lifting problem.
\newblock {\em Math. Ann.}, 345(3):711--730, 2009.

\bibitem[BWZ09]{BWZ:dd}
Irene~I. Bouw, Stefan Wewers, and Leonardo Zapponi.
\newblock Deformation data, {B}elyi maps, and the local lifting problem.
\newblock {\em Trans. Amer. Math. Soc.}, 361(12):6645--6659, 2009.

\bibitem[Bys11]{By:df}
Jakub Byszewski.
\newblock Deformation functors of local actions.
\newblock Preprint, arXiv:1112.0352, 2011.

\bibitem[CGH08]{CGH:og}
Ted Chinburg, Robert Guralnick, and David Harbater.
\newblock Oort groups and lifting problems.
\newblock {\em Compos. Math.}, 144(4):849--866, 2008.

\bibitem[CGH11]{CGH:ll}
Ted Chinburg, Robert Guralnick, and David Harbater.
\newblock The local lifting problem for actions of finite groups on curves.
\newblock {\em Ann. Sci. \'Ec. Norm. Sup\'er. (4)}, 44(4):537--605, 2011.

\bibitem[CGH15]{CGH:go}
Ted Chinburg, Robert Guralnick, and David Harbater.
\newblock Global {O}ort groups.
\newblock Preprint, arXiv:1512.09112, 2015.

\bibitem[CK03]{CK:ed}
Gunther Cornelissen and Fumiharu Kato.
\newblock Equivariant deformation of {M}umford curves and of ordinary curves in
  positive characteristic.
\newblock {\em Duke Math. J.}, 116(3):431--470, 2003.

\bibitem[CK05]{CK:es}
Gunther Cornelissen and Fumiharu Kato.
\newblock Zur {E}ntartung schwach verzweigter {G}ruppenoperationen auf
  {K}urven.
\newblock {\em J. Reine Angew. Math.}, 589:201--236, 2005.

\bibitem[CM06]{CM:rr}
Gunther Cornelissen and Ariane M{\'e}zard.
\newblock Rel\`evements des rev\^etements de courbes faiblement ramifi\'es.
\newblock {\em Math. Z.}, 254(2):239--255, 2006.

\bibitem[CO05]{AutQuestions}
Gunther Cornelissen and Frans Oort.
\newblock Problems from the workshop on automorphisms of curves ({L}eiden,
  {A}ugust, 2004).
\newblock {\em Rendiconti del Seminario Matematico della Universit\`{a} di
  Padova}, 113:129--177, 2005.

\bibitem[Eis95]{Ei:ca}
David Eisenbud.
\newblock {\em Commutative algebra}, volume 150 of {\em Graduate Texts in
  Mathematics}.
\newblock Springer-Verlag, New York, 1995.
\newblock With a view toward algebraic geometry.

\bibitem[FGI{\etalchar{+}}05]{FGAexplained}
Barbara Fantechi, Lothar G{\"o}ttsche, Luc Illusie, Steven~L. Kleiman, Nitin
  Nitsure, and Angelo Vistoli.
\newblock {\em Fundamental algebraic geometry}, volume 123 of {\em Mathematical
  Surveys and Monographs}.
\newblock American Mathematical Society, Providence, RI, 2005.
\newblock Grothendieck's FGA explained.

\bibitem[FM02]{FM:cs}
Michael~D. Fried and Ariane M{\'e}zard.
\newblock Configuration spaces for wildly ramified covers.
\newblock In {\em Arithmetic fundamental groups and noncommutative algebra
  ({B}erkeley, {CA}, 1999)}, volume~70 of {\em Proc. Sympos. Pure Math.}, pages
  353--376. Amer. Math. Soc., Providence, RI, 2002.

\bibitem[Ful69]{Fu:hs}
William Fulton.
\newblock Hurwitz schemes and irreducibility of moduli of algebraic curves.
\newblock {\em Ann. of Math. (2)}, 90:542--575, 1969.

\bibitem[Gar96]{Ga:pr}
Marco~A. Garuti.
\newblock Prolongement de rev\^etements galoisiens en g\'eom\'etrie rigide.
\newblock {\em Compos. Math.}, 104(3):305--331, 1996.

\bibitem[Gar02]{Ga:ls}
Marco~A. Garuti.
\newblock Linear systems attached to cyclic inertia.
\newblock In {\em Arithmetic fundamental groups and noncommutative algebra
  ({B}erkeley, {CA}, 1999)}, volume~70 of {\em Proc. Sympos. Pure Math.}, pages
  377--386. Amer. Math. Soc., Providence, RI, 2002.

\bibitem[GM98]{GM:lg}
Barry Green and Michel Matignon.
\newblock Liftings of {G}alois covers of smooth curves.
\newblock {\em Compos. Math.}, 113(3):237--272, 1998.

\bibitem[GM99]{GM:op}
Barry Green and Michel Matignon.
\newblock Order {$p$} automorphisms of the open disc of a {$p$}-adic field.
\newblock {\em J. Amer. Math. Soc.}, 12(1):269--303, 1999.

\bibitem[Gre03]{Gr:af}
Barry Green.
\newblock Automorphisms of formal power series rings over a valuation ring.
\newblock In {\em Valuation theory and its applications, {V}ol. {II}
  ({S}askatoon, {SK}, 1999)}, volume~33 of {\em Fields Inst. Commun.}, pages
  79--87. Amer. Math. Soc., Providence, RI, 2003.

\bibitem[Gre04]{Gr:dc}
Barry Green.
\newblock Realizing deformations of curves using {L}ubin-{T}ate formal groups.
\newblock {\em Israel J. Math.}, 139:139--148, 2004.

\bibitem[Gro95]{gfga}
Alexander Grothendieck.
\newblock G\'eom\'etrie formelle et g\'eom\'etrie alg\'ebrique.
\newblock In {\em S\'eminaire {B}ourbaki, {V}ol.\ 5}, pages Exp.\ No.\ 182,
  193--220, errata p.\ 390. Soc. Math. France, Paris, 1995.

\bibitem[Har77]{Ha:ag}
Robin Hartshorne.
\newblock {\em Algebraic geometry}.
\newblock Springer-Verlag, New York-Heidelberg, 1977.
\newblock Graduate Texts in Mathematics, No. 52.

\bibitem[Har80]{Ha:mp}
David Harbater.
\newblock Moduli of {$p$}-covers of curves.
\newblock {\em Comm. Algebra}, 8(12):1095--1122, 1980.

\bibitem[Har94]{Ha:ac}
David Harbater.
\newblock Abhyankar's conjecture on {G}alois groups over curves.
\newblock {\em Invent. Math.}, 117(1):1--25, 1994.

\bibitem[Har03]{Ha:pg}
David Harbater.
\newblock Patching and {G}alois theory.
\newblock In {\em Galois groups and fundamental groups}, volume~41 of {\em
  Math. Sci. Res. Inst. Publ.}, pages 313--424. Cambridge Univ. Press,
  Cambridge, 2003.

\bibitem[Hen00a]{He:ht}
Yannick Henrio.
\newblock Arbres de hurwitz et automorphismes d'ordre $p$ des disques et des
  couronnes $p$-adiques formels.
\newblock Preprint, arXiv:math/0011098, 2000.

\bibitem[Hen00b]{He:dc}
Yannick Henrio.
\newblock Disques et couronnes ultram\'etriques.
\newblock In {\em Courbes semi-stables et groupe fondamental en g\'eom\'etrie
  alg\'ebrique ({L}uminy, 1998)}, volume 187 of {\em Progr. Math.}, pages
  21--32. Birkh\"auser, Basel, 2000.

\bibitem[Hen02]{He:thesis}
Yannick Henrio.
\newblock Th\`{e}se, Universit\'{e} Bordeaux I, available at
  https://www.math.u-bordeaux.fr/~mmatigno/Henrio-These.pdf, 2002.

\bibitem[Kat81]{Ka:ST}
N.~Katz.
\newblock Serre-{T}ate local moduli.
\newblock In {\em Algebraic surfaces ({O}rsay, 1976--78)}, volume 868 of {\em
  Lecture Notes in Math.}, pages 138--202. Springer, Berlin-New York, 1981.

\bibitem[Kat86]{Ka:lg}
Nicholas~M. Katz.
\newblock Local-to-global extensions of representations of fundamental groups.
\newblock {\em Ann. Inst. Fourier (Grenoble)}, 36(4):69--106, 1986.

\bibitem[Kat88]{Ka:TL}
Nicholas~M. Katz.
\newblock Travaux de {L}aumon.
\newblock {\em Ast\'erisque}, (161-162):Exp.\ No.\ 691, 4, 105--132 (1989),
  1988.
\newblock S{\'e}minaire Bourbaki, Vol. 1987/88.

\bibitem[LL78]{LL:dc}
O.~A. Laudal and K.~L{\o}nsted.
\newblock Deformations of curves. {I}. {M}oduli for hyperelliptic curves.
\newblock In {\em Algebraic geometry ({P}roc. {S}ympos., {U}niv. {T}roms\o,
  {T}roms\o, 1977)}, volume 687 of {\em Lecture Notes in Math.}, pages
  150--167. Springer, Berlin, 1978.

\bibitem[Mat99]{Ma:pg}
Michel Matignon.
\newblock {$p$}-groupes ab\'eliens de type {$(p,\cdots,p)$} et disques ouverts
  {$p$}-adiques.
\newblock {\em Manuscripta Math.}, 99(1):93--109, 1999.

\bibitem[MRT14]{MRT:ss}
Ariane M{\'e}zard, Matthieu Romagny, and Dajano Tossici.
\newblock Sekiguchi-{S}uwa theory revisited.
\newblock {\em J. Th\'eor. Nombres Bordeaux}, 26(1):163--200, 2014.

\bibitem[Mum99]{Mu:rb}
David Mumford.
\newblock {\em The red book of varieties and schemes}, volume 1358 of {\em
  Lecture Notes in Mathematics}.
\newblock Springer-Verlag, Berlin, expanded edition, 1999.
\newblock Includes the Michigan lectures (1974) on curves and their Jacobians,
  With contributions by Enrico Arbarello.

\bibitem[Nak87]{Na:pr}
Sh{\=o}ichi Nakajima.
\newblock {$p$}-ranks and automorphism groups of algebraic curves.
\newblock {\em Trans. Amer. Math. Soc.}, 303(2):595--607, 1987.

\bibitem[NO80]{NO:ma}
Peter Norman and Frans Oort.
\newblock Moduli of abelian varieties.
\newblock {\em Ann. of Math. (2)}, 112(3):413--439, 1980.

\bibitem[Obu12]{Ob:ll}
Andrew Obus.
\newblock The (local) lifting problem for curves.
\newblock In {\em Galois-{T}eichm\"uller theory and arithmetic geometry},
  volume~63 of {\em Adv. Stud. Pure Math.}, pages 359--412. Math. Soc. Japan,
  Tokyo, 2012.

\bibitem[Obu15]{Ob:go}
Andrew Obus.
\newblock A generalization of the {O}ort conjecture.
\newblock \emph{Comm.\ Math.\ Helv.}, to appear, arXiv:1502.07623, 2015.

\bibitem[Obu16]{Ob:A4}
Andrew Obus.
\newblock The local lifting problem for ${A}_4$.
\newblock {\em Algebra Number Theory}, 10(8):1683--1693, 2016.

\bibitem[Oor87]{Oo:la}
Frans Oort.
\newblock Lifting algebraic curves, abelian varieties, and their endomorphisms
  to characteristic zero.
\newblock In {\em Algebraic geometry, {B}owdoin, 1985 ({B}runswick, {M}aine,
  1985)}, volume~46 of {\em Proc. Sympos. Pure Math.}, pages 165--195. Amer.
  Math. Soc., Providence, RI, 1987.

\bibitem[OW14]{OW:ce}
Andrew Obus and Stefan Wewers.
\newblock Cyclic extensions and the local lifting problem.
\newblock {\em Ann. of Math. (2)}, 180(1):233--284, 2014.

\bibitem[OW16]{OW:wr}
Andrew Obus and Stefan Wewers.
\newblock Wild ramification kinks.
\newblock {\em Res. Math. Sci.}, 3:3:21, 2016.

\bibitem[Pag02a]{Pa:ev}
Guillaume Pagot.
\newblock {${\Bbb F}\sb p$}-espaces vectoriels de formes diff\'erentielles
  logarithmiques sur la droite projective.
\newblock {\em J. Number Theory}, 97(1):58--94, 2002.

\bibitem[Pag02b]{Pa:rc}
Guillaume Pagot.
\newblock Rel\`{e}vement en caract\'{e}ristique z\'{e}ro d'actions de groupes
  ab\'{e}liens de type $(p, \ldots, p)$.
\newblock Th\`{e}se, Universit\'{e} Bordeaux I, available at
  http://www.math.u-bordeaux1.fr/~mmatigno/Pagot-These.pdf, 2002.

\bibitem[Pop14]{Po:oc}
Florian Pop.
\newblock The {O}ort conjecture on lifting covers of curves.
\newblock {\em Ann. of Math. (2)}, 180(1):285--322, 2014.

\bibitem[Pri02]{Pr:fw}
Rachel~J. Pries.
\newblock Families of wildly ramified covers of curves.
\newblock {\em Amer. J. Math.}, 124(4):737--768, 2002.

\bibitem[PZ12]{PZ:pr}
Rachel Pries and Hui~June Zhu.
\newblock The {$p$}-rank stratification of {A}rtin-{S}chreier curves.
\newblock {\em Ann. Inst. Fourier (Grenoble)}, 62(2):707--726, 2012.

\bibitem[Ray94]{Ra:ab}
Michel Raynaud.
\newblock Rev\^etements de la droite affine en caract\'eristique {$p>0$} et
  conjecture d'{A}bhyankar.
\newblock {\em Invent. Math.}, 116(1-3):425--462, 1994.

\bibitem[Rob77]{Ro:ct}
Abraham Robinson.
\newblock {\em Complete theories}.
\newblock North-Holland Publishing Co., Amsterdam-New York-Oxford, second
  edition, 1977.
\newblock With a preface by H. J. Keisler, Studies in Logic and the Foundations
  of Mathematics.

\bibitem[Roq70]{Ro:aa}
Peter Roquette.
\newblock Absch\"atzung der {A}utomorphismenanzahl von {F}unktionenk\"orpern
  bei {P}rimzahlcharakteristik.
\newblock {\em Math. Z.}, 117:157--163, 1970.

\bibitem[Sa{\"{\i}}12]{Sa:fl}
Mohamed Sa{\"{\i}}di.
\newblock Fake liftings of {G}alois covers between smooth curves.
\newblock In {\em Galois-{T}eichm\"uller theory and arithmetic geometry},
  volume~63 of {\em Adv. Stud. Pure Math.}, pages 457--501. Math. Soc. Japan,
  Tokyo, 2012.

\bibitem[Sch68]{Sc:fa}
Michael Schlessinger.
\newblock Functors of {A}rtin rings.
\newblock {\em Trans. Amer. Math. Soc.}, 130:208--222, 1968.

\bibitem[Ser61]{Se:ev}
Jean-Pierre Serre.
\newblock Exemples de vari\'et\'es projectives en caract\'eristique {$p$} non
  relevables en caract\'eristique z\'ero.
\newblock {\em Proc. Nat. Acad. Sci. U.S.A.}, 47:108--109, 1961.

\bibitem[Ser79]{Se:lf}
Jean-Pierre Serre.
\newblock {\em Local fields}, volume~67 of {\em Graduate Texts in Mathematics}.
\newblock Springer-Verlag, New York-Berlin, 1979.
\newblock Translated from the French by Marvin Jay Greenberg.

\bibitem[SGA03]{SGA1}
{\em Rev\^etements \'etales et groupe fondamental ({SGA} 1)}.
\newblock Documents Math\'ematiques (Paris) [Mathematical Documents (Paris)],
  3. Soci\'et\'e Math\'ematique de France, Paris, 2003.
\newblock S{\'e}minaire de g{\'e}om{\'e}trie alg{\'e}brique du Bois Marie
  1960--61. [Algebraic Geometry Seminar of Bois Marie 1960-61], Directed by A.
  Grothendieck, With two papers by M. Raynaud, Updated and annotated reprint of
  the 1971 original [Lecture Notes in Math., 224, Springer, Berlin; MR0354651
  (50 \#7129)].

\bibitem[SOS89]{SOS:ask}
T.~Sekiguchi, F.~Oort, and N.~Suwa.
\newblock On the deformation of {A}rtin-{S}chreier to {K}ummer.
\newblock {\em Ann. Sci. \'Ecole Norm. Sup. (4)}, 22(3):345--375, 1989.

\bibitem[SS94]{SS:kasw1}
Tsutomu Sekiguchi and Noriyuki Suwa.
\newblock On the unified {K}ummer-{A}rtin-{S}chreier-{W}itt theory.
\newblock Chuo Math. preprint series, 1994.

\bibitem[SS99]{SS:kasw2}
Tsutomu Sekiguchi and Noriyuki Suwa.
\newblock On the unified {K}ummer-{A}rtin-{S}chreier-{W}itt theory.
\newblock Laboratoire de Math\'{e}matiques Pures de Bordeaux preprint series,
  1999.

\bibitem[Sza09]{Sz:gg}
Tam{\'a}s Szamuely.
\newblock {\em Galois groups and fundamental groups}, volume 117 of {\em
  Cambridge Studies in Advanced Mathematics}.
\newblock Cambridge University Press, Cambridge, 2009.

\bibitem[Tur15]{Tu:el}
Daniele Turchetti.
\newblock Equidistant liftings of elementary abelian galois covers of curves.
\newblock Preprint, arxiv:1510.06453, 2015.

\bibitem[Wea]{We:D4}
Bradley Weaver.
\newblock The local lifting problem for ${D}_4$.
\newblock In preparation.

\bibitem[Wew99]{We:dt}
Stefan Wewers.
\newblock Deformation of tame admissible covers of curves.
\newblock In {\em Aspects of {G}alois theory ({G}ainesville, {FL}, 1996)},
  volume 256 of {\em London Math. Soc. Lecture Note Ser.}, pages 239--282.
  Cambridge Univ. Press, Cambridge, 1999.

\bibitem[Zap08]{Za:1p}
Leonardo Zapponi.
\newblock On the 1-pointed curves arising as \'etale covers of the affine line
  in positive characteristic.
\newblock {\em Math. Z.}, 258(4):711--727, 2008.

\end{thebibliography}

\end{document}